\documentclass[12pt]{amsart}  

\usepackage{amsmath,amsthm,amssymb,amscd}
\usepackage[all]{xy}
\usepackage{setspace}
\allowdisplaybreaks

\newtheorem{theorem}{Theorem}[section]
\newtheorem{proposition}[theorem]{Proposition}
\newtheorem{definition}[theorem]{Definition}
\newtheorem{lemma}[theorem]{Lemma}
\newtheorem{corollary}[theorem]{Corollary}
\newtheorem{conjecture}[theorem]{Conjecture}

\numberwithin{equation}{section}

\newtheorem{gauss}{Lemma}

\DeclareFontEncoding{OT2}{}{}
\DeclareFontSubstitution{OT2}{cmr}{m}{n}
\DeclareFontFamily{OT2}{cmr}{\hyphenchar\font45 }
\DeclareFontShape{OT2}{cmr}{m}{n}{<5><6><7><8><9>gen*wncyr<10><10.95><12><14.4><17.28><20.74><24.88>wncyr10}{}
\DeclareFontShape{OT2}{cmr}{b}{n}{<5><6><7><8><9>gen*wncyb<10><10.95><12><14.4><17.28><20.74><24.88>wncyb10}{}
\DeclareMathAlphabet{\mathcyr}{OT2}{cmr}{m}{n}
\DeclareMathAlphabet{\mathcyb}{OT2}{cmr}{b}{n}
\SetMathAlphabet{\mathcyr}{bold}{OT2}{cmr}{b}{n}
\newcommand{\sh}{\mathcyr {sh}}

\newcommand{\R}{\mathbb R}
\newcommand{\Z}{{\mathbb Z}}

\newcommand{\C}{{\mathbb C}}
\newcommand{\Q}{{\mathbb Q}}
\newcommand{\odd}{\mathbb O}

\newcommand{\cusp}{\mathbb S}
\newcommand{\Zsp}{\mathcal Z}
\newcommand{\Set}{{\mathrm S}}
\newcommand{\V}{{\mathsf{Vect}}}
\newcommand{\W}{{\mathbf W}}
\newcommand{\m}{{\mathfrak m}}
\newcommand{\depth}{{\mathfrak D}}
\newcommand{\stuffle}{{\Psi}}
\newcommand{\inj}{{\Phi}}
\newcommand{\tr}[1]{{}^t\hspace{0mm}#1}
\newcommand{\cp}{\, \underline{\circ}\,}
\newcommand{\ec}{e\tbinom{m_1,\ldots,m_r}{n_1,\ldots,n_r}}
\newcommand{\cc}{c{\tbinom{m_1,\ldots,m_r}{n_1,\ldots,n_r}}}
\newcommand{\rank}{{\rm rank}\,}

\title[Multiple zeta values and period polynomials]{On linear relations among totally odd multiple zeta values related to period polynomials}
\author[K.~Tasaka]{Koji Tasaka}
\thanks{This work is partially supported by Japan Society for the Promotion of Science, Grant-in-Aid for JSPS Fellows (No. 241440) and Basic Science Research Program through the National Research Foundation of Korea (NRF) funded by the Ministry of Education (2013053914)}
\address[Koji Tasaka]{Graduate School of Mathematics, Nagoya University, Japan}
\email{koji.tasaka@math.nagoya-u.ac.jp}

\date{}

\begin{document}
\maketitle

\begin{abstract}
We show that there is a relationship between modular forms and totally odd multiple zeta values, by relating the matrix $E_{N,r}$, whose entries are given by the polynomial representations of the Ihara action, with even period polynomials.
We also consider the matrix $C_{N,r}$ defined by Brown and give a new upper bound of the rank of $C_{N,4}$.
This result gives support to the uneven part of the motivic Broadhurst-Kreimer conjecture of depth 4.
\end{abstract}

%%%%%%%%%%%%%%%%%%%%%%%%%%%%%%%%%%%%%%%%%%%
\section{Introduction}

In this paper, we will be interested in the relationship between $\Q$-linear relations among multiple zeta values 
\[ \zeta(n_1,\ldots,n_r) = \sum_{0<k_1<\cdots<k_r} \frac{1}{k_1^{n_1}\cdots k_r^{n_r}} \quad (n_1,\ldots,n_{r-1}\in \Z_{>0}, n_r\in\Z_{>1}) \] 
(MZVs for short) and period polynomials of the elliptic modular form for the full modular group $\Gamma_1:={\rm PSL}_2 (\Z)$.
This interesting connection was first discovered in the case of depth 2 by Zagier~\cite[\S8]{Z} (see also \cite[\S3]{Z1}), and then investigated in depth by Gangl, Kaneko and Zagier~\cite{GKZ}. 
We will investigate this connection for arbitrary depths through the study of totally odd MZVs which are MZVs at the sequences indexed by odd integers greater than 1, modulo all MZVs of lower depth and the ideal generated by $\zeta(2)$.

The totally odd MZV was first considered by Francis Brown~\cite[\S10]{B} in his study of the depth-graded motivic MZV
\[ \zeta^{\m}_{\depth}(n_1,\ldots,n_r) \]
which is the motivic MZV $\zeta^{\m}(n_1,\ldots,n_r)$ modulo lower depth.
He conjectures that the universal enveloping algebra of the Lie subalgebra $\mathfrak{g}^{\rm odd}$ of the depth-graded motivic Lie algebra $\mathfrak{dg}^\m$ (see \cite[\S4]{B}) generated by the canonical elements $\sigma_{2i+1} \ (i\ge1)$ of degree $-2i-1$ and of depth $1$ is dual to the space ${\rm gr}^{\depth}\mathcal{A}^{\rm odd}$ spanned by all totally odd motivic MZVs.
Assuming his conjecture \cite[Conjecture~4]{B} which recasts the motivic version of the Broadhurst-Kreimer conjecture~\cite{BK} as a statement of the homology of $\mathfrak{dg}^\m$, one can compute the dimension of $\mathfrak{g}^{\rm odd}$, and hence a conjectural dimension formula for totally odd (motivic) MZVs is obtained.
This dimension formula is called the uneven part of the (motivic) Broadhurst-Kreimer conjecture (see Conjecture~\ref{2_1} and \cite[Conjecture~5]{B}).

We wish to study the uneven part of the (motivic) Broadhurst-Kreimer conjecture.
An important framework has been done by Brown \cite[\S10]{B}.
This uses the notion of his celebrated paper~\cite{B1}: Brown's operator $D_m$, which is an infinitesimal version of the motivic coaction on the motivic MZVs, preserves the depth filtration $\depth$ of motivic MZVs (we will give a brief review of his notion in \S2.3).
Using Brown's operator, we define a certain matrix $C_{N,r}$ (Definition~2.3).
It is highly expected by definition that all relations among totally odd (motivic) MZVs of weight $N$ and depth $r$ are obtained from right annihilators of $C_{N,r}$, which is compatible with \cite[Conjecture~4]{B} and not known for $r\ge4$ (see Conjecture~\ref{comb}).
As the size of the square matrix $C_{N,r}$ coincides with the number of totally odd MZVs of weight $N$ and depth $r$, we are led to the following conjecture.
\begin{conjecture}\label{1_1}
The rank of the matrix $C_{N,r}$ is equal to the dimension of the $\Q$-vector space spanned by all totally odd (motivic) MZVs of weight $N$ and depth $r$.
\end{conjecture}

This paper examines the matrix $C_{N,r}$, and makes an attempt to compute the rank of them.
The coefficients of $C_{N,r}$ are computable, but it is hard to give an exact value of $\rank C_{N,r}$ in general.
Some information about it can be obtained by looking at left annihilators of $C_{N,r}$: for example, all left annihilators of $C_{N,2}$ are characterised by restricted even period polynomials of degree $N-2$, which was first shown by Baumard and Schneps~\cite{BS}.
Here the restricted even period polynomial of degree $N-2$ is defined as an even, homogeneous polynomial $p(x_1,x_2)\in \Q[x_1,x_2]_{(N-2)}$ of degree $N-2$ such that $p(x_1,0)=0$ and
\begin{equation}\label{eq1_1}   
p(x_1,x_2) - p(x_2-x_1,x_2)+p(x_2-x_1,x_1)=0.
\end{equation}
The important fact about this polynomial is that the dimension of its associated vector space over $\Q$ coincides with the dimension of the space of cusp forms for $\Gamma_1$, which follows from, known as Eichler-Shimura-Manin correspondence, the results of \cite[\S1.1]{KZ} and \cite[\S5]{GKZ}.
Therefore, since the matrix $C_{N,r}$ is square, the above characterisation provides a lower bound of the dimension of the space of right annihilators of $C_{N,2}$, and also an upper bound of $\rank C_{N,2}$.
It is vital to note that relating with period polynomials is actually our fundamental principle of computing an upper bound of $\rank C_{N,r}$.

Our first goal (\S3) is to relate the restricted even period polynomials and a certain matrix $E_{N,r}$ (Definition~\ref{3_2}). % whose coefficients are given by the polynomial representations of the Ihara action (or from another point of view by the Fourier coefficients of the multiple Eisenstein series), and whose coefficients are simply written in terms of binomial coefficients. 
The square matrix $E_{N,r}$ is a right factor of the matrix $C_{N,r}$ (Proposition~\ref{3_3}), which shows
\begin{equation}\label{ineq}
\dim_{\Q} \ker E_{N,r}\le \dim_{\Q} \ker C_{N,r}.
\end{equation}
We prove that there is an injection from the space of restricted even period polynomials to the space of left annihilators of $E_{N,r}$, and, as a consequence, a lower bound of $\dim_{\Q} \ker E_{N,r}$ is obtained.
More precisely, let ${\bf W}_{N,r}$ be the space of even, homogeneous polynomials $p(x_1,\ldots,x_r)\in \Q[x_1,\ldots,x_r]_{(N-r)}$ of degree $N-r$ such that $p(x_1,\ldots,x_r)|_{x_i=0}=0$ for $1\le i\le r$ and 
\begin{equation*}
 p(x_1,\ldots,x_r) - p(x_2-x_1,x_2,\underbrace{x_3,\ldots,x_r}_{r-2})+p(x_2-x_1,x_1,\underbrace{x_3,\ldots,x_r}_{r-2})=0.
\end{equation*}
Each element in $\W_{N,2}$ is just a restricted even period polynomial, and the defining relation in the space $\W_{N,r}$ only uses the relation \eqref{eq1_1}, so that the dimension of the space of $\W_{N,r}$ is easily obtained.
Our first main result is as follows (Proposition~\ref{3_5} and Theorem~\ref{3_7}).
\begin{theorem}\label{1_3}
(1) There is a one-to-one correspondence between the space $\W_{N,2}$ and the space of left annihilators of the matrix $E_{N,2}(=C_{N,2})$.\\
(2) For $r\ge3$, there is an injection from the space $\W_{N,r}$ to the space of left annihilators of the matrix $E_{N,r}$.
\end{theorem}

Theorem~\ref{1_3} (1) is the result of Baumard and Schneps, which will be reproved in \S3.3 for completeness.
The map in the statement of Theorem~\ref{1_3} (2) is quite simple and guessed from numerical experiments (so we have no idea to explain this map theoretically).
It is interesting to note that this map is conjecturally surjective, i.e., all left annihilators of $E_{N,r}$ can be characterised by restricted even period polynomials.
We were not able to prove this conjecture.

The second main result of this paper is a new upper bound of $\rank C_{N,4}$ which is the predicted bound (see \eqref{eq2_5}).
We remark that $\rank C_{N,2}$ is known (Theorem~\ref{1_3} (1)), and it is easy to deduce the predicted upper bound of $\rank C_{N,3}$ which we do not discuss in this paper.

From \eqref{ineq} and Theorem~\ref{1_3} one obtains a lower bound of $\dim_{\Q} \ker C_{N,r}$, but for $r\ge3$ it is not the predicted bound.
To obtain more elements in the space $\ker C_{N,r}$, we use an algebraic structure of the space of totally odd MZVs.
This space forms the commutative $\Q$-algebra with respect to Hoffman's harmonic product (modulo lower depth).
%Since the matrix $E_{N,r}$ is a right factor of the matrix $C_{N,r}$, it is expected that the linear relation obtained from the right annihilator of $E_{N',r'}$ for $(N',r')<(N,r)$, multiplying totally odd MZVs of weight $N-N'$ and depth $r-r'$, provides the right annihilator of $C_{N,r}$.
This algebraic structure yields an injection from the space of right annihilators of $E_{N',r'}$ for $(N',r')<(N,4)$ to the space of right annihilators of $C_{N,4}$ (Proposition~\ref{4_2}).
Using these results, we obtain the predicted upper bound of $\rank C_{N,4}$.
\begin{theorem}\label{1_4}
We have
\[ \sum_{N>0} \rank C_{N,4} x^N \le \odd(x)^4-3\odd(x)^2 \cusp(x) +\cusp(x)^2,\]
where $\odd(x)=\frac{x^3}{1-x^2},\ \cusp(x)=\frac{x^{12}}{(1-x^4)(1-x^6)}$ and $\sum_{N>0} a_N x^N \le \sum_{N>0} b_N x^N$ means $a_N\le b_N$ for all $N>0$.
\end{theorem}

Theorem~\ref{1_4} gives support to the uneven part of the motivic Broadhurst-Kreimer conjecture of depth 4.
In fact, if Conjecture~\ref{1_1} (or equivalently Conjecture~\ref{comb}) for $r=4$ is true, then one obtains the predicted upper bound of the dimension of the $\Q$-vector space spanned by totally odd MZVs of depth 4. 
It is our project in progress to give a solution of Conjecture~\ref{1_1} for $r=4$.
 % because it is believed that there are generators of the depth-graded motivic Lie algebra $\mathfrak{dg}^\m$ in depth 4 which are ``uneven" (see \cite[Section~8.3 and 8.4]{B} and \cite[Section~8.3]{B2}). 
%\begin{corollary}\label{1_5}
%We have
%\[ \sum_{N>0} \dim_{\Q}\Zsp_{N,4}^{\rm odd} x^N \le \odd(x)^4-3\odd(x)^2 \cusp(x) +\cusp(x)^2.\]
%\end{corollary}

The contents are as follows.
In Section 2, we state the uneven part of the Broadhurst-Kreimer conjecture, and give a brief review of Brown's works in \cite{B}.
Section 3 studies the matrix $E_{N,r}$. 
We shall give a proof of Theorem~\ref{1_3}.
Section 4 is devoted to proving Theorem~\ref{1_4}.
This proof needs a certain lemma about the shuffle algebra, which will be proven in the Appendix.

%%%%%%%%%%%%%%%%%%%%%%%%%%%%%%%%%%%%%%%%%%
\section{Preliminaries}

%%%%%%%%%%%%%%%%%%%%%%%%%%%%%%%%%%%%%%%%%%
\subsection{Totally odd multiple zeta values conjecture}

The MZV is defined for ${\bf n} =(n_1,\ldots,n_r) \in \Z_{>0}^r$ with $n_r\ge2$ by 
\[ \zeta({\bf n})=\zeta(n_1,\ldots,n_r) = \sum_{0<k_1<\cdots<k_r} \frac{1}{k_1^{n_1}\cdots k_r^{n_r}}. \] 
As usual, we call $n_1+\cdots+n_r \ (=:{\rm wt}({\bf n}))$ the weight and $r\ (=:{\rm dep}({\bf n}))$ the depth, and $1\in\Q$ is regarded as the unique MZV of weight 0 and depth 0.
Let $\Zsp$ be the MZV algebra $\bigoplus_{N\ge0}\Zsp_N$, where $\Zsp_N$ is the $\Q$-vector space spanned by all MZVs of weight $N$, and $\depth$ the depth filtration on $\Zsp$:
\[ \depth_0 \Zsp =\Q \subset \depth_1\Zsp \subset \cdots \subset \depth_r\Zsp := \langle \zeta({\bf n}) \mid {\rm dep}({\bf n})\le r \rangle_{\Q} \subset \cdots .\]
The MZV algebra becomes a filtered algebra with $\depth$.
Let $\Zsp_{N,r}$ be the $\Q$-vector space of the weight $N$ and depth $r$ part of the bigraded $\Q$-algebra ${\rm gr}^{\depth} \big(\Zsp/\zeta(2)\Zsp\big)=\bigoplus_{N,r\ge0} \Zsp_{N,r}$ and $\zeta_{\depth}({\bf n})$ denote the equivalence class of $\zeta({\bf n})$ of weight $N$ and depth $r$ in $\Zsp_{N,r}$, called the depth-graded MZV.
We notice that the $\Q$-vector space $\Zsp_{N,r}=\depth_r\Zsp_N\big/ \big( \depth_{r-1}\Zsp_N + \depth_r\Zsp_N\cap \zeta(2)\Zsp \big)$ is spanned by all depth-graded MZVs of weight $N$ and depth $r$.
Let us call $\zeta_{\depth}(n_1,\ldots,n_r)$ the totally odd MZV when all $n_i$ are odd $(\ge3)$. 
The $\Q$-vector subspace of $\Zsp_{N,r}$ spanned by all totally odd MZVs of weight $N$ and depth $r$ is denoted by
\[ \Zsp_{N,r}^{\rm odd} = \big\langle \zeta_{\depth} ({\bf n})\in \Zsp_{N,r} \mid {\bf n}\in \Set_{N,r} \big\rangle_{\Q}, \]
where $\Set_{N,r}$ is the set of totally odd indices of weight $N$ and depth $r$:
\[ \Set_{N,r} = \{ (n_1,\ldots,n_r)\in \Z^r \mid n_1+\cdots+n_r=N,\ n_1,\ldots,n_r\ge3:{\rm odd} \}. \]
We set $\Zsp_{0,0}^{\rm odd}=\Q$.
Notice that the number of elements of the set $\Set_{N,r}$ obviously gives a trivial upper bound $\dim_{\Q} \Zsp_{N,r}^{\rm odd} \le |\Set_{N,r}|$, and hence $\Zsp_{N,r}^{\rm odd}=\{0\}$ whenever $N\not\equiv r \pmod{2}$.
We now state the uneven part of the Broadhurst-Kreimer conjecture.
\begin{conjecture}\label{2_1}{\rm (\cite[Eq.~(10.4)]{B})}
The generating function of the dimension of the space $\Zsp_{N,r}^{\rm odd}$ is given by
\begin{equation}\label{eq2_1}
\sum_{N,r\ge0} \dim_{\Q} \Zsp_{N,r}^{\rm odd} x^N y^r \stackrel{?}{=} \frac{1}{1-\odd (x) y +\cusp (x) y^2},
\end{equation}
where $\odd(x)=\frac{x^3}{1-x^2}=x^3+x^5+x^7+\cdots$ and $\cusp (x)= \frac{x^{12}}{(1-x^4)(1-x^6)}=x^{12}+x^{16}+x^{18}+\cdots$.
\end{conjecture}

Let us give a few examples. 
By expanding the right-hand side of \eqref{eq2_1} at $y=0$, one can derive
\begin{align*}
& \sum_{N>0} \dim_{\Q} \Zsp_{N,2}^{\rm odd} x^N\stackrel{?}{=} \odd(x)^2-\cusp(x),\ \sum_{N>0} \dim_{\Q} \Zsp_{N,3}^{\rm odd} x^N\stackrel{?}{=} \odd(x)^3-2\odd(x) \cusp(x), \\
& \sum_{N>0} \dim_{\Q} \Zsp_{N,4}^{\rm odd} x^N\stackrel{?}{=} \odd(x)^4-3\odd(x)^2 \cusp(x) +\cusp(x)^2
\end{align*}
and so on.
Since $\odd (x)^r=\sum_{N>0}|\Set_{N,r}|x^N$ and the coefficient of $x^N$ in $\cusp (x)$ coincides with the dimension of the space of cusp forms of weight $N$ for $\Gamma_1$, Conjecture~\ref{2_1} suggests that all linear relations among totally odd MZVs relate to cusp forms.
In the cases of $r=2$ and $3$, by the results of Goncharov~\cite[Theorems~2.4 and 2.5]{G} (see also \cite[Proposition~18]{IKZ}, \cite[Theorem~19]{IO}) on the number of algebra generators of the depth-graded $\Q$-algebra ${\rm gr}^{\depth}\big( \Zsp\big/ \zeta(2)\Zsp \big)$, we find
\[ \sum_{N>0} \dim_{\Q} \Zsp_{N,2}^{\rm odd} x^N \le \odd (x)^2-\cusp(x)\quad \mbox{and}\quad \sum_{N>0} \dim_{\Q} \Zsp_{N,3}^{\rm odd} x^N \le \odd(x)^3-2\odd(x)\cusp(x).\]

{\it Remark}.
To obtain the above inequalities, we have used $\dim_{\Q}\Zsp_{N,r}^{\rm odd} \le \dim_{\Q}\Zsp_{N,r}$.
For $r=2$ it follows from the result of Gangl, Kaneko and Zagier~\cite[Theorem~2]{GKZ} that $\Zsp^{\rm odd}_{N,2}=\Zsp_{N,2}$.
However we do not know $\Zsp_{N,3}^{\rm odd} \stackrel{?}{=} \Zsp_{N,3}$.

%%%%%%%%%%%%%%%%%%%%%%%%%%%%%%%%%%%%%%%%%%
\subsection{Notations}
To clarify the meaning of the left (or right) annihilator of our matrices described in the Introduction, we fix notations.
For integers $N,r$ with $|\Set_{N,r}|>0$, the notation 
$$M=\left( m\tbinom{m_1,\ldots,m_r}{n_1,\ldots,n_r} \right)_{\begin{subarray}{c} (m_1,\ldots,m_r)\in \Set_{N,r}\\ (n_1,\ldots,n_r)\in \Set_{N,r}\end{subarray}}$$
means that $M$ is a $|\Set_{N,r}|\times |\Set_{N,r}|$ matrix with $m\tbinom{m_1,\ldots,m_r}{n_1,\ldots,n_r}\in\Z$ in the $((m_1,\ldots,m_r)$, $(n_1,\ldots,n_r))$ entry (i.e. rows and columns are indexed by $(m_1,\ldots,m_r)$ and $(n_1,\ldots,n_r)$ in the set $\Set_{N,r}$ respectively).
For convenience we regard $M$ as an empty matrix when $|\Set_{N,r}|=0$ (i.e. $\rank M=0$).
Let us denote by $\V_{N,r}$ the $|\Set_{N,r}|$-dimensional vector space over $\Q$ of row vectors $(a_{n_1,\ldots,n_r})_{(n_1,\ldots,n_r)\in \Set_{N,r}}$ indexed by totally odd indices $(n_1,\ldots,n_r)\in \Set_{N,r}$ with rational coefficients:
\[ \V_{N,r} = \{ (a_{n_1,\ldots,n_r})_{(n_1,\ldots,n_r)\in \Set_{N,r}} \mid a_{n_1,\ldots,n_r}\in \Q \}.\]
We identify the matrix $M$ with its induced linear map on $\V_{N,r}$
\begin{align*}
M : \V_{N,r}&\longrightarrow \V_{N,r}\\
 v&\longmapsto v\cdot M,
\end{align*}
so that, for $v=(a_{n_1,\ldots,n_r})_{(n_1,\ldots,n_r)\in \Set_{N,r}} \in \V_{N,r}$
\[ M\big(v \big) = \Big( \sum_{(m_1,\ldots,m_r)\in \Set_{N,r}} a_{m_1,\ldots,m_r} m\tbinom{m_1,\ldots,m_r}{n_1,\ldots,n_r} \Big)_{(n_1,\ldots,n_r)\in \Set_{N,r}}. \]
It is clear that a row vector $v=(a_{n_1,\ldots,n_r})_{(n_1,\ldots,n_r)\in \Set_{N,r}} \in \V_{N,r}$ satisfies $M(v)=0$ if and only if $\sum_{(m_1,\ldots,m_r)\in \Set_{N,r}} a_{m_1,\ldots,m_r} m\binom{m_1,\ldots,m_r}{n_1,\ldots,n_r} =0$ for all $(n_1,\ldots,n_r)\in \Set_{N,r}$.
Hereafter, we use the notion $v\in \ker M\subset \V_{N,r}$ (resp. $\ker\tr M\subset\V_{N,r}$) instead of saying $v$ (resp. $\tr v$) is a left (resp. right) annihilator of the matrix $M$, where we understand $\tr M= \left( m{\tbinom{n_1,\ldots,n_r}{m_1,\ldots,m_r}} \right)_{\begin{subarray}{c} (m_1,\ldots,m_r)\in \Set_{N,r}\\ (n_1,\ldots,n_r)\in \Set_{N,r}\end{subarray}}$.

%%%%%%%%%%%%%%%%%%%%%%%%%%%%%%%%%%%%%%%%%%%%%%%%%%%%%%
\subsection{Linear relations among totally odd MZVs}

The motivic multiple zeta value $\zeta^{\m}(n_1,\ldots,n_r)$ plays a key role in detecting Conjecture~\ref{2_1}.
This is an element of a certain $\Q$-algebra $\mathcal{H}=\bigoplus_{N\ge0}\mathcal{H}_N$ constructed by Brown~\cite[Definition~2.1]{B1}, which uses an idea of Goncharov~\cite{G1}, and it has the period map $per:\mathcal{H}\rightarrow \R$ mapping the elements $\zeta^{\m}(n_1,\ldots,n_r)$ to the real number $\zeta(n_1,\ldots,n_r)$.
We remark that $\zeta^\m(2)$ is not treated to be zero in $\mathcal{H}$.
Denote by $\depth$ the depth filtration  (see \cite[Section 4]{B})
\[ \depth_r\mathcal{H} = \langle \zeta^{\m} ({\bf n}) \mid {\rm dep}({\bf n})\le r \rangle_{\Q},\] 
and by $\zeta_{\depth}^{\m}({\bf n})$ the depth-graded motivic MZV which is an image of $\zeta^{\m}({\bf n})$ in ${\rm gr}^{\depth} \mathcal{H}$.
Let $ \mathcal{H}_{N,r}^{\rm odd}$ be the $\Q$-vector subspace of ${\rm gr}^{\depth}_r \mathcal{H}_N$ spanned by all totally odd motivic MZVs $\zeta_{\depth}^{\m}({\bf n})$ of weight $N$ and depth $r$ (i.e. ${\bf n}\in \Set_{N,r}$):
\[ \mathcal{H}_{N,r}^{\rm odd} := \langle \zeta^\m_{\depth} ({\bf n}) \in {\rm gr}^{\depth}_r \mathcal{H}_N \mid {\bf n} \in \Set_{N,r} \rangle_{\Q}.\]
According to the result of Brown~\cite[Proposition 10.1]{B}, for each odd integer $m>1$ one can define a well-defined derivation $\partial_m$ such that
\[ \partial_m :  \mathcal{H}_{N,r}^{\rm odd} \longrightarrow  \mathcal{H}_{N-m,r-1}^{\rm odd} \ \mbox{and} \  \partial_m\big(\zeta^{\m}_{\depth}(n)\big) = \delta\tbinom{n}{m} ,\]
where $\delta\tbinom{n}{m}=1$ if $n=m$ and 0 otherwise.
We sketch an explicit construction of the derivation $\partial_m$ (we follow an idea used in \cite[Section~5]{B1}).

Recall Brown's operator $D_m$ (see \cite[Definition~3.1]{B1}).
This is an infinitesimal coaction obtained by the coaction of the algebra-comodule $\mathcal{H}$, and becomes a derivation (i.e. $D_m(\xi_1\xi_2)=(1\otimes \xi_1) D_m(\xi_2) + (1\otimes \xi_2) D_m(\xi_1)$ for $\xi_1,\xi_2\in\mathcal{H}$).
The derivation $D_m$ can be computed by using the following explicit formula: for odd $m>1$
\begin{equation}\label{key1}
\begin{aligned}
&D_m(I^\m(a_0;a_1,\ldots,a_N;a_{N+1}))\\
&= \sum_{p=0}^{N-n} I^{\mathfrak{L}} (a_p;a_{p+1},\ldots,a_{p+n};a_{p+n+1})\otimes I^\m(a_0;a_1,\ldots,a_p,a_{p+n+1},\ldots,a_N;a_{N+1}),
\end{aligned}
\end{equation}
where for $a_i\in\{0,1\}$ the element $I^\m(a_0;a_1,\ldots,a_N;a_{N+1})\in\mathcal{H}_N$ is the motivic iterated integral (see \cite[(2.16)]{B1}) and $I^{\mathfrak{L}}$ is an image of $I^\m$ in $\mathcal{L}=\mathcal{A}_{>0}/(\mathcal{A}_{>0})^2$ where $\mathcal{A}=\mathcal{H}/\zeta^{\m}(2)\mathcal{H}$.
We remark that the motivic iterated integral defines the motivic multiple zeta value:
\[ \zeta^\m(n_1,\ldots,n_r) = I^\m(0;1,\underbrace{0,\ldots,0}_{n_1-1},\ldots,1,\underbrace{0,\ldots,0}_{n_r-1};1).\]
Let us denote by $\mathcal{L}_{m}$ the weight $m$ part of the Lie coalgebra $\mathcal{L}$ and by $\mathcal{H}^{\rm odd}$ the $\Q$-vector space spanned by the set $\{\zeta^\m(n_1,\ldots,n_r)\mid r\ge0,n_i\ge3:{\rm odd}\}$.
From the result \cite[Proposition~10.1]{B} we have for odd $m>1$
\[ D_m \big( \depth_r \mathcal{H}^{\rm odd} \big) \subset \mathcal{L}_m \otimes \depth_{r-1} \mathcal{H}^{\rm odd} + \mathcal{L}_m \otimes \depth_{r-2} \mathcal{H},\]
which gives a map for $r\in\Z_{>0}$
\[ {\rm gr}_r^{\depth}D_{m} : \mathcal{H}^{\rm odd}_{N,r} \longrightarrow \mathcal{L}_{m} \otimes_{\Q} \mathcal{H}_{N-m,r-1}^{\rm odd} .\]
By computing ${\rm gr}_r^{\depth} D_m \big( \zeta^{\m}_{\depth} ({\bf n}) \big)$ for ${\bf n}\in \Set_{N,r}$ (see an explicit formula in \eqref{eq4_17}), one can find 
\begin{equation*}
 {\rm gr}_r^{\depth}D_{m}\big(  \mathcal{H}^{\rm odd}_{N,r} \big) \subset \Q\zeta_{m}\otimes_{\Q} \mathcal{H}^{\rm odd}_{N-m,r-1},
\end{equation*}
where $\zeta_{m}$ is an image of $\zeta^{\m}(m)$ in $\mathcal{L}$.
Then the above derivation $\partial_m $ is defined to be $(\zeta_{m}^{\vee}\otimes 1) \circ {\rm gr}^{\depth}_r D_m\big|_{ \mathcal{H}_{N,r}^{\rm odd}}$.

The operator $\partial_m$ corresponds to the action of the canonical generators of the depth-graded motivic Lie algebra $\mathfrak{dg}^\m$ (for the definition, see \cite[\S4]{B}) in depth 1 .
If one believes the Broadhurst-Kreimer conjecture, there should be extra generators of the depth-graded motivic Lie algebra $\mathfrak{dg}^\m$ in only depth 4 and these generators should act trivially on totally odd motivic MZVs.
Thus one is led to the following conjecture.

\begin{conjecture}\label{comb}
A $\Q$-linear combination $\xi$ of totally odd motivic MZVs of weight $N$ and depth $r$ is zero if and only if it satisfies $\partial_{m_r}\circ \partial_{m_{r-1}}\circ \cdots \circ \partial_{m_1} (\xi) =0$ for all $(m_1,\ldots,m_r)\in \Set_{N,r}$.
\end{conjecture}

Conjecture~\ref{comb} is in fact true when $r=2$ and $3$, although we omit the proof.
To prove Conjecture~\ref{comb} for $r=4$, we need to show that the elements $\mathfrak{c}(P)$, obtained by Brown~\cite[Theorem~8.5]{B2}, are uneven modulo commutators of the canonical generators of $\mathfrak{dg}^\m$ in depth 1 (or that the exceptional elements $\overline{\bf e}$ also given by Brown~\cite[Definition~8.1]{B} are motivic), which was pointed out by Francis Brown.
We refer to \cite[\S8.3]{B2} and \cite[\S8.3]{B} for further details.

Conjecture~\ref{comb} is equivalent to Conjecture~\ref{1_1} described in the Introduction. 
This follows from the definition of the matrix $C_{N,r}$ given below.
For totally odd indices $(n_1,\ldots,n_r),(m_1,\ldots,m_r)\in \Set_{N,r}$ we set
\begin{equation}\label{eq2_2}
c\tbinom{m_1,\ldots,m_r}{n_1,\ldots,n_r} = \partial_{m_r}\circ \partial_{m_{r-1}}\circ \cdots \circ \partial_{m_1} \big( \zeta^{\m}_{\depth} (n_1,\ldots,n_r)\big) \in \Q.
\end{equation} 
We note that $c\tbinom{m}{n}=\delta\tbinom{m}{n}$.
\begin{definition}\label{2_2}
For integers $N>r>0$, we define the $|\Set_{N,r}|\times|\Set_{N,r}|$ matrix $C_{N,r}$ by
\begin{equation*}
C_{N,r} = \left( c\tbinom{m_1,\ldots,m_r}{n_1,\ldots,n_r} \right)_{\begin{subarray}{c} (m_1,\ldots,m_r) \in \Set_{N,r} \\ (n_1,\ldots,n_r)\in \Set_{N,r} \end{subarray}}.
\end{equation*}
\end{definition}

Assuming Conjecture~\ref{comb}, we find that the relation
\[ \sum_{(n_1,\ldots,n_r)\in \Set_{N,r}} a_{n_1,\ldots,n_r}\zeta^{\m}_{\depth}(n_1,\ldots,n_r) = 0  \]
holds if and only if $(a_{n_1,\ldots,n_r})_{(n_1,\ldots,n_r)\in \Set_{N,r}}\in \ker\tr C_{N,r}$.
Thus, for $N>r>0$ we have
\begin{equation}\label{eq2_4}
\dim_{\Q} \mathcal{H}_{N,r}^{\rm odd} \stackrel{?}{=} \rank C_{N,r}=|S_{N,r}|-\dim_{\Q}\ker C_{N,r},
\end{equation}
which is Conjecture~\ref{1_1}.
Computing the generating series of $\rank C_{N,r}$ (up to $N=30$), one obtains the following conjecture (see \cite[\S10]{B}):
\begin{equation}\label{eq2_5}
1+\sum_{N>r>0} {\rm rank}\ C_{N,r} x^Ny^r \stackrel{?}{=} \frac{1}{1-\odd(x)y+\cusp(x) y^2} ,
\end{equation}
which by \eqref{eq2_4} implies the uneven part of the motivic Broadhurst-Kreimer conjecture (Conjecture~\ref{2_1}).

%%%%%%%%%%%%%%%%%%%%%%%%%%%%%%%%%%%%%%%%%%
\section{Linear relation among totally odd MZVs and even period polynomials}

%%%%%%%%%%%%%%%%%%%%%%%%%%%%%%%%%%%%%%%%%%%%%%%%%%%%%%
\subsection{Polynomial representations of Ihara action}
We recall the polynomial representation of the Ihara action of Brown~\cite[\S6]{B}.
This provides an expression of the generating function of the integers $\cc$.

A polynomial representation of the depth-graded version of the linearised Ihara action $\cp:\Q[x_1,\ldots,x_r]\otimes_{\Q} \Q[x_1,\ldots,x_s] \rightarrow \Q[x_1,\ldots,x_{r+s}]$ is given explicitly by 
\begin{equation}\label{eq3_1}
\begin{aligned}
&f\cp g (x_1,\ldots,x_{r+s})=\sum_{i=0}^s f(x_{i+1}-x_i,\ldots,x_{i+r}-x_i)g(x_1,\ldots,x_i,x_{i+r+1},\ldots,x_{r+s})\\
&+(-1)^{{\rm deg}(f)+r}\sum_{i=1}^s f(x_{i+r-1}-x_{i+r},\ldots,x_i-x_{i+r})g(x_1,\ldots,x_{i-1},x_{i+r},\ldots,x_{r+s})
\end{aligned}
\end{equation}
for homogeneous polynomials $f(x_1,\ldots,x_r)$ and $g(x_1,\ldots,x_s)$, where $x_0=0$.
We note that, by duality, for totally odd indices $(m_1,\ldots,m_r),(n_1,\ldots,n_r)\in \Set_{N,r}$ the number $\cc$ defined in \eqref{eq2_2} coincides with the coefficient of $x_1^{n_1-1}\cdots x_r^{n_r-1}$ in $x_1^{m_1-1}\cp (\cdots x_{1}^{m_{r-2}-1}\cp(x_{1}^{m_{r-1}-1}\cp x_1^{m_r-1}) \cdots)$, i.e.
\begin{equation}\label{eq3_2} 
\begin{aligned}
x_1^{m_1-1}\cp &(\cdots x_{1}^{m_{r-2}-1}\cp(x_{1}^{m_{r-1}-1}\cp x_1^{m_r-1}) \cdots) \\
&= \sum_{\substack{n_1+\cdots+n_r=m_1+\cdots+m_r\\n_1,\ldots,n_r\ge1}}\cc x_1^{n_1-1}\cdots x_r^{n_r-1}.
\end{aligned}
\end{equation}

For integers $m_1,\ldots,m_r,n_1,\ldots,n_r\ge1$, let $\ec$ be the integer obtained from the coefficient of $x_1^{n_1-1}\cdots x_r^{n_r-1}$ in $x_1^{m_1-1}\cp (x_1^{m_2-1}\cdots x_{r-1}^{m_r-1})$:
\begin{equation}\label{eq3_3}
x_1^{m_1-1}\cp (x_1^{m_2-1}\cdots x_{r-1}^{m_r-1})=\sum_{\substack{n_1+\cdots+n_r=m_1+\cdots +m_r\\ n_1,\ldots,n_r\ge1}} \ec x_1^{n_1-1}\cdots x_r^{n_r-1} ,
\end{equation}
and $ e\tbinom{m_1}{n_1}=\delta\tbinom{m_1}{n_1}$.
For the latter purpose we now give explicit formulas of the integers $\ec$.
For integers $m_1,\ldots,m_r,n_1,\ldots,n_r,m,n,n'\ge1$ let us define $\delta{\tbinom{m_1,\ldots,m_r}{n_1,\ldots,n_r}}$ as the Kronecker delta given by
\[ \delta{\tbinom{m_1,\ldots,m_r}{n_1,\ldots,n_r}}= \begin{cases} 1 & \mbox{if $m_i=n_i$ for all $i\in \{ 1,\ldots,r\}$}\\
0 & \mbox{otherwise} \end{cases}, \]
and the integer $b_{n,n'}^m$ by
$$b_{n,n'}^{m}=(-1)^{n}\binom{m-1}{n-1}+(-1)^{n'-m} \binom{m-1}{n'-1}.$$
It is obvious that for odd $n,n',m>1$ one has 
\begin{equation}\label{eqan}
b_{n,n'}^m+b_{n',n}^m=0.
\end{equation}
\begin{lemma}\label{3_1}
For integers $m_1,\ldots,m_r,n_1,\ldots,n_r\ge1$, we have
\begin{equation*}
\ec= \delta{\tbinom{m_1,\ldots,m_r}{n_1,\ldots,n_r}}+ \sum_{i=1}^{r-1} \delta{\tbinom{m_2,\ldots,m_i,m_{i+2},\ldots,m_r}{n_1,\ldots,n_{i-1},n_{i+2},\ldots,n_r}} b_{n_{i},n_{i+1}}^{m_1}.
\end{equation*}
\end{lemma}
\begin{proof}
By definition \eqref{eq3_1}, one can compute
\begin{align*}
&x_1^{m_1-1}\cp (x_1^{m_2-1}\cdots x_{r-1}^{m_r-1})= x_1^{m_1-1}\cdots x_r^{m_r-1}  \\
&+ \sum_{i=1}^{r-1} \big( x_{i+1}-x_i\big)^{m_1-1}  \big( x_1^{m_2-1} \cdots x_i^{m_{i+1}-1} x_{i+2}^{m_{i+2}-1}\cdots x_r^{m_r-1} \\
&- x_1^{m_2-1} \cdots x_{i-1}^{m_{i}-1} x_{i+1}^{m_{i+1}-1}\cdots x_r^{m_r-1} \big)\\
&=x_1^{m_1-1}\cdots x_r^{m_r-1} +\sum_{i=1}^{r-1} x_1^{m_2-1}\cdots x_{i-1}^{m_i-1} x_{i+2}^{m_{i+2}-1} \cdots x_r^{m_r-1} \\
& \times \sum_{\substack{n_i+n_{i+1}=m_1+m_{i+1}\\ n_i,n_{i+1}\ge1}}\left( (-1)^{m_1-n_{i+1}} \binom{m_1-1}{n_{i+1}-1}- (-1)^{n_i-1} \binom{m_1-1}{n_i-1} \right)  x_i^{n_{i}-1} x_{i+1}^{n_{i+1}-1} \\
&= \sum_{\substack{n_1+\cdots+n_r=m_1+\cdots+m_r\\n_i\ge1}}\bigg( \delta{\tbinom{m_1,\ldots,m_r}{n_1,\ldots,n_r}}+ \sum_{i=1}^{r-1} \delta{\tbinom{m_2,\ldots,m_i,m_{i+2},\ldots,m_r}{n_1,\ldots,n_{i-1},n_{i+2},\ldots,n_r}} b_{n_{i},n_{i+1}}^{m_1}\bigg) x_1^{n_1-1}\cdots x_r^{n_r-1},
\end{align*}
which completes the proof.
\end{proof}

{\it Remark}. 
One interesting thing about the integer $\ec$ is that it can be obtained as a part of coefficients of the Fourier expansion of the multiple Eisenstein series.
In general, the author and Bachmann~\cite{BT} showed the correspondence between the Fourier expansion of multiple Eisenstein series and the coproduct $\Delta$ defined by Goncharov~\cite{G1}.
However, we were not able to prove any linear relations among totally odd MZVs from this correspondence whenever $r\ge3$.

\

We end this section by defining the matrix $E_{N,r}$ we are actually interested in.
\begin{definition}\label{3_2}
For $N>r>0$, we define the $|\Set_{N,r}|\times |\Set_{N,r}|$ matrix $ E_{N,r}$ by
\begin{equation*}
 E_{N,r} = \left( \ec \right)_{\begin{subarray}{c} (m_1,\ldots,m_r)\in \Set_{N,r}\\ (n_1,\ldots,n_r)\in \Set_{N,r}\end{subarray}} .
\end{equation*}
\end{definition}

%%%%%%%%%%%%%%%%%%%%%%%%%%%%%%%%%%%%%%%%%%
\subsection{The relation between $E_{N,r}$ and $C_{N,r}$}

For integers $r\ge2$ and $2\le q\le r$, let $E_{N,r}^{(q)}$ be the $|\Set_{N,r}|\times |\Set_{N,r}|$ matrix defined by
\begin{equation}\label{eq3_4}
 E_{N,r}^{(q)} = \left( \delta\tbinom{m_1,\ldots,m_{r-q}}{n_1,\ldots,n_{r-q}}\cdot  e\tbinom{m_{r-q+1},\ldots,m_r}{n_{r-q+1},\ldots,n_r}  \right)_{\begin{subarray}{c} (m_1,\ldots,m_r)\in \Set_{N,r} \\ (n_1,\ldots,n_r) \in \Set_{N,r}  \end{subarray}}.
\end{equation}
We note $E_{N,r}^{(r)}=E_{N,r}$.
We now prove that the matrix $E_{N,r}$ is a right factor of the matrix $C_{N,r}$, which was indicated by Seidai Yasuda.

\begin{proposition}\label{3_3}
For integers $N>r>1$, we have
\begin{equation*}
 C_{N,r} = E_{N,r}^{(2)} \cdot E_{N,r}^{(3)}   \cdots  E_{N,r}^{(r-1)} \cdot E_{N,r}.
\end{equation*}
\end{proposition}
\begin{proof}
This is shown by induction on $r$.
When $r=2$, the assertion $C_{N,2}=E_{N,2}$ follows from the definition of the integer $\ec$ in \eqref{eq3_3} and the generating function of $\cc$ in \eqref{eq3_2}.
For $r\ge3$, we set
\begin{align*}
f(x_1,\ldots,x_{r-1}) &= x_1^{m_2-1}\cp ( \cdots \cp(x_{1}^{m_{r-1}-1}\cp x_{1}^{m_r-1} )\cdots)\\
&=\sum_{\substack{n_2+\cdots+n_r=m_2+\cdots+m_r\\n_2,\ldots,n_r\ge1}} c{\tbinom{m_2,\ldots,m_r}{n_2,\ldots,n_r}} x_1^{n_2-1}\cdots x_{r-1}^{n_r-1}.
\end{align*}
By definition, we see that $x_1^{m_1-1}\cp f(x_1,\ldots,x_{r-1})  = \sum_{n_1+\cdots+n_r=N} \cc x_1^{n_1-1}\cdots x_r^{n_r-1}$, where $N=m_1+\cdots+m_r$.
One can also compute by linearity
\begin{align*}
x_1^{m_1-1}\cp f(x_1,\ldots,x_{r-1}) &= \sum_{t_2+\cdots+t_r=m_2+\cdots+m_r} c{\tbinom{m_2,\ldots,m_r}{t_2,\ldots,t_r}} x_1^{m_1-1}\cp ( x_1^{t_2-1}\cdots x_{r-1}^{t_r-1})\\
&=  \sum_{t_1+\cdots+t_r=N}\delta\tbinom{m_1}{t_1} c{\tbinom{m_2,\ldots,m_r}{t_2,\ldots,t_r}} x_1^{t_1-1}\cp  (x_1^{t_2-1}\cdots x_{r-1}^{t_r-1})\\
&=  \sum_{t_1+\cdots+t_r=N} \delta\tbinom{m_1}{t_1} c{\tbinom{m_2,\ldots,m_r}{t_2,\ldots,t_r}}  \sum_{n_1+\cdots+n_r=N}  e{\tbinom{t_1,\ldots,t_r}{n_1,\ldots,n_r}} x_1^{n_1-1}\cdots x_r^{n_r-1}\\
&= \sum_{n_1+\cdots+n_r=N} \bigl( \sum_{t_1+\cdots+t_r=N} \delta\tbinom{m_1}{t_1} c{\tbinom{m_2,\ldots,m_r}{t_2,\ldots,t_r}}   e{\tbinom{t_1,\ldots,t_r}{n_1,\ldots,n_r}} \bigr) x_1^{n_1-1}\cdots x_r^{n_r-1},
\end{align*}
where $t_1,\ldots,t_r\ge1$.
For the totally odd indices $(m_1,\ldots,m_r),(n_1,\ldots,n_r)\in \Set_{N,r}$, the term $ \delta\tbinom{m_1}{t_1}  e{\tbinom{t_1,\ldots,t_r}{n_1,\ldots,n_r}}$ in the last equation is 0 if $(t_1,\ldots,t_r)\not\in \Set_{N,r}$ with $t_1+\cdots+t_r=N$ (this follows from the explicit formula of $\ec$ in Lemma~\ref{3_1}), so one has
\begin{equation*}
c{\tbinom{m_1,\ldots,m_{r}}{n_1,\ldots,n_{r}}} = \sum_{(t_1,\ldots,t_r)\in \Set_{N,r}} \delta\tbinom{m_1}{t_1}c{\tbinom{m_2,\ldots,m_{r}}{t_2,\ldots,t_{r}}}   e{\tbinom{t_1,\ldots,t_r}{n_1,\ldots,n_r}}.
\end{equation*}
Then our claim follows from the induction hypothesis
\[ \left( \delta\tbinom{m_1}{n_1} c\tbinom{m_2,\ldots,m_r}{n_2,\ldots,n_r} \right)_{\begin{subarray}{c} (m_1,\ldots,m_r)\in \Set_{N,r} \\ (n_1,\ldots,n_r)\in \Set_{N,r} \end{subarray}} = E_{N,r}^{(2)} \cdot E_{N,r}^{(3)}   \cdots  E_{N,r}^{(r-1)}.\]
\end{proof}

Let us illustrate a few examples of the elements in $\ker\tr E_{N,r}$. 
For the matrix $E_{N,2}$, the first non-trivial example is obtained from the matrix 
\[
E_{12,2}=\left(
\begin{array}{cccc}
  e{\tbinom{3,9}{3,9}} &   e{\tbinom{3,9}{5,7}} &   e{\tbinom{3,9}{7,5}} &   e{\tbinom{3,9}{9,3}} \\
  e{\tbinom{5,7}{3,9}} &   e{\tbinom{5,7}{5,7}} &   e{\tbinom{5,7}{7,5}} &   e{\tbinom{5,7}{9,3}}  \\
  e{\tbinom{7,9}{5,9}} &   e{\tbinom{7,5}{5,7}} &   e{\tbinom{7,5}{7,5}} &   e{\tbinom{7,5}{9,3}}  \\
  e{\tbinom{9,3}{3,9}} &   e{\tbinom{9,3}{5,7}} &   e{\tbinom{9,3}{7,5}} &   e{\tbinom{9,3}{9,3}}  \\
\end{array}
\right) =\left(
\begin{array}{cccc}
 0 & 0 & 0 & 1 \\
 -6 & 0 & 1 & 6 \\
 -15 & -14 & 15 & 15 \\
 -27 & -42 & 42 & 28 \\
\end{array}
\right) .
\]
The space $\ker\tr E_{12,2}$ is generated by the vector $(14, 75, 84, 0)$, which gives the well-known relation obtained by Gangl, Kaneko and Zagier~\cite{GKZ}:
\begin{equation}\label{eg1}
14\zeta_{\depth}(3,9)+75\zeta_{\depth}(5,7)+84\zeta_{\depth}(7,5) =0 .
\end{equation}
In the case of $r=3$, we consider the matrix 
\begin{align*}
&E_{15,3} =\left(
\begin{array}{cccccccccc}
 0 & 0 & 0 & 1 & 0 & 0 & 0 & 0 & 0 & 0 \\
 0 & 0 & 0 & 0 & 0 & 0 & 1 & 0 & 0 & 0 \\
 0 & 0 & 0 & 0 & 0 & 0 & 0 & 0 & 1 & 0 \\
 0 & 0 & 0 & 0 & 0 & 0 & 0 & 0 & 0 & 1 \\
 -6 & -6 & 1 & 6 & 6 & 0 & 0 & 0 & 0 & 0 \\
 0 & 0 & -6 & 0 & -6 & 1 & 6 & 6 & 0 & 0 \\
 0 & 0 & 0 & -6 & 0 & 0 & 0 & -5 & 6 & 6 \\
 -15 & -14 & 0 & 15 & 0 & 0 & 0 & 15 & 0 & 0 \\
 0 & 0 & 0 & -15 & -14 & 0 & 0 & 0 & 15 & 15 \\
 -27 & -42 & 42 & 0 & 0 & 0 & -42 & 0 & 42 & 28 \\
\end{array}
\right).
\end{align*}
The space $\ker\tr E_{15,3}$ is generated by $(-14, 15, 6, 0, 0, 36, 0, 0, 0, 0)$, which gives 
\begin{equation}\label{eg2}
-14 \zeta_{\depth}(3, 3, 9) + 15 \zeta_{\depth}(3, 5, 7) + 6 \zeta_{\depth}(3, 7, 5) +  36 \zeta_{\depth}(5, 5, 5)=0 .
\end{equation}
The following relation is obtained from the space $\ker\tr E_{18,4}$ by assuming Conjecture~\ref{comb} for $r=4$: 
\begin{equation}\label{eg3}
70 \zeta_{\depth}(3, 3, 3, 9) - 75 \zeta_{\depth}(3, 3, 5, 7) - 30 \zeta_{\depth}(3, 3, 7, 5) +  36 \zeta_{\depth} (3, 5, 5, 5)=0.
\end{equation}

{\it Remark}.
It is probably worth mentioning that the elements in $\ker\tr E_{N,2}$ are completely determined by restricted even period polynomials of degree $N-2$, first discovered by Gangl, Kaneko and Zagier~\cite{GKZ} (see also \cite{BS}).
As a natural question, one can ask if there is an explicit description of elements in $\ker\tr E_{N,r}$ for $r>2$.

%%%%%%%%%%%%%%%%%%%%%%%%%%%%%%%%%%%%%%%%%%
\subsection{Relation with period polynomials}
This section is devoted to proving Theorem~\ref{1_3}.

Let $\mathbf{P}_{N,r}\subset\Q[x_1,\ldots,x_r]$ be the $|\Set_{N,r}|$-dimensional $\Q$-vector space spanned by the set $\{x_1^{n_1-1}\cdots x_r^{n_r-1}\mid (n_1,\ldots,n_r)\in \Set_{N,r}\}$, and $\W_{N,r}$ its subspace defined for $r\ge2$ by
\begin{equation*}
\W_{N,r}=\{ p\in \mathbf{P}_{N,r}\mid p(x_1,\ldots,x_r) = p(x_2-x_1,x_2,\underbrace{x_3,\ldots,x_r}_{r-2})-p(x_2-x_1,x_1,\underbrace{x_3,\ldots,x_r}_{r-2})\}.
\end{equation*}
The space $\W_{N,2}$ is called the space of restricted even period polynomials (see \cite{BS}).
As mentioned in the Introduction, the dimension of the space $\W_{N,2}$ is equal to the dimension of the $\C$-vector space of cusp forms of weight $N$ for $\Gamma_1$, so that we have
\begin{equation}\label{eq3_5}
 \sum_{N>0} \dim_{\Q} \W_{N,2} x^N =\cusp(x).
\end{equation}
For $r\ge3$, we easily find that $\W_{N,r}\cong \bigoplus_{1<n<N} \W_{n,2} \otimes_{\Q} \mathbf{P}_{N-n,r-2}$.
In fact, every element in $\W_{N,r}$ can be written as $\Q$-linear combinations of the form $p(x_1,x_2) x_3^{n_3-1}\cdots x_r^{n_r-1}$ ($p(x_1,x_2)\in \W_{n,2}, (n_3,\ldots,n_r)\in \Set_{N-n,r-2}, 1<n<N$).
Thus, from $\odd(x)^r=\sum_{N>0} |\Set_{N,r}|x^N$ and \eqref{eq3_5}, we obtain the dimension of the space $\W_{N,r}$:
\begin{equation}\label{eq3_6}
\sum_{N>0} \dim_{\Q} \W_{N,r} x^N = \cusp(x)\cdot \odd(x)^{r-2}.
\end{equation}

Baumard and Schneps~\cite{BS} have shown that the space $\W_{N,2}$ is isomorphic to the space $\ker E_{N,2}$.
Since this fact is used for proving Theorem~\ref{1_4}, we reprove it for completeness.
Set an isomorphism $\pi(=\pi^{(N,r)})$ as follows.
\begin{align*}
\pi:{\mathbf P}_{N,r}&\longrightarrow \V_{N,r}\\
 \sum_{(n_1,\ldots,n_r)\in \Set_{N,r}} a_{n_1,\ldots,n_r} x_1^{n_1-1}\cdots x_r^{n_r-1}&\longmapsto (a_{n_1,\ldots,n_r})_{(n_1,\ldots,n_r)\in \Set_{N,r}}.
 \end{align*}

\begin{proposition}\label{3_5} {\rm (\cite[Proposition~3.2]{BS})} 
For each integer $N>0$ one has
\[ \pi \big( \W_{N,2} \big) = \ker E_{N,2}.\]
\end{proposition}
\begin{proof}
For every polynomial $p(x_1,x_2)=\sum_{(n_1,n_2)\in \Set_{N,2}} a_{n_1,n_2} x_1^{n_1-1}x_2^{n_2-1}$ satisfying $p(x,x)=0$, one can compute
\begin{align}
\notag&p(x_1,x_2)-p(x_2-x_1,x_2)+p(x_2-x_1,x_1)\\
\notag&= \sum_{(m_1,m_2)\in \Set_{N,2}} a_{m_1,m_2} \sum_{\substack{n_1+n_2=N\\n_1,n_2\ge2}} \bigg( \delta\tbinom{m_1,m_2}{n_1,n_2}- (-1)^{n_1-1} \binom{m_1-1}{n_1-1} \\
\notag&+ (-1)^{m_1-n_2}\binom{m_1-1}{n_2-1}\bigg) x_1^{n_1-1}x_2^{n_2-1} \\
\notag&=  \sum_{\substack{n_1+n_2=N\\n_1,n_2\ge2}} \Big(   \sum_{(m_1,m_2)\in \Set_{N,2}} a_{m_1,m_2}  e{\tbinom{m_1,m_2}{n_1,n_2}} \Big) x_1^{n_1-1}x_2^{n_2-1}\\
\label{eq3_7}&= \sum_{(n_1,n_2)\in \Set_{N,2}} \Big( \sum_{(m_1,m_2)\in \Set_{N,2}} a_{m_1,m_2}  e{\tbinom{m_1,m_2}{n_1,n_2}} \Big) x_1^{n_1-1}x_2^{n_2-1}\\
\label{eq3_8}&+\frac{1}{2}\big( p(x_2-x_1,x_1)-p(x_2-x_1,x_2)-p(x_2+x_1,x_1)+p(x_2+x_1,x_2)\big).
\end{align}
Assume $p(x_1,x_2)\in \W_{N,2}$.
Since $p(x_1,x_2)=p(x_2-x_1,x_2)-p(x_2-x_1,x_1)=p(x_2+x_1,x_2)-p(x_2+x_1,x_1)$, the polynomial in \eqref{eq3_8} is zero.
Then the polynomial in \eqref{eq3_7} has to be 0, which implies $\pi(p)\in \ker E_{N,2}$.
We now prove $\pi^{-1} (v)(x_1,x_2) \in \W_{N,2}$ for $v\in \ker E_{N,2}$ using the action of the group ${\rm PGL}_2(\Z)$ on $\Q[x_1,x_2]_{(N-2)}$ $(N:{\rm even})$ defined by $(f\big|\gamma)(x_1,x_2) = f(ax_1+bx_2,cx_1+dx_2)$ for $\gamma=(\begin{smallmatrix} a&b \\c&d \end{smallmatrix})$.
Set
\[ \delta = \begin{pmatrix} -1&0\\0&1 \end{pmatrix},\  \varepsilon = \begin{pmatrix} 0&1\\1&0 \end{pmatrix}, \ T = \begin{pmatrix} 1&1\\0&1 \end{pmatrix} \in {\rm PGL}_2(\Z).\]
For $v=(a_{n_1,n_2})_{(n_1,n_2)\in \Set_{N,2}}\in \ker E_{N,2}$, one has
\[0=E_{N,2}(v)= \bigg( a_{n_1,n_2} + \sum_{(m_1,m_2)\in \Set_{N,2}} a_{m_1,m_2} b_{n_1,n_2}^{m_1} \bigg)_{(n_1,n_2)\in \Set_{N,2}} .\]
By \eqref{eqan}, one obtains $a_{n_1,n_2}=-a_{n_2,n_1}$.
This shows $\pi^{-1}(v)\big| ( \varepsilon +1)=0$, where we have extended the action of ${\rm PGL}_2(\Z)$ to its group ring by linearity.
We notice that $\pi^{-1}(v)\big| (\delta -1)=0$ because of even.
Using $T\delta =\delta T^{-1}$ and $T \varepsilon \delta =  \varepsilon T \varepsilon T^{-1}$, we have $ \pi^{-1}(v)\big|(1-T+T \varepsilon )\delta = \pi^{-1}(v)\big|(1-T^{-1} +  \varepsilon T  \varepsilon T^{-1} ) = -\pi^{-1}(v)\big| (1-T+T \varepsilon) T^{-1} $.
Let $G=\pi^{-1}(v)\big|(1-T+T \varepsilon )$. 
Then $G(0,x_2)=0$ and 
\[ 0=2\times \eqref{eq3_7} = G\big|(1+\delta)=G\big| (1-T^{-1}).\]
Since the coefficient matrix obtained by the action of $1-T^{-1}$ on $G$ becomes triangular, one finds $G=0$.
The assertion follows from
\begin{align*} 0=&G(x_1,x_2)= \pi^{-1}(v)(x_1,x_2)-\pi^{-1}(v)(x_2+x_1,x_2)+\pi^{-1}(v)(x_2+x_1,x_1)\\
=&\pi^{-1}(v)(x_1,x_2)-\pi^{-1}(v)(x_2-x_1,x_2)+\pi^{-1}(v)(x_2-x_1,x_1).
\end{align*}
\end{proof}

From \eqref{eq3_5} one immediately obtains an exact value of $\rank E_{N,2}$.
\begin{corollary}\label{3_6} The generating function of the rank of the matrix $E_{N,2}$ is given by
\begin{equation*}
\sum_{N>0} \rank E_{N,2}x^N= \odd(x)^2-\cusp(x),
\end{equation*}
or equivalently, one has $\sum_{N>0} \dim_{\Q} \ker E_{N,2}x^N = \cusp(x)$. 
\end{corollary}

We now present our result for $r\ge3$.
Set the identity map $I_{N,r}\in M_{|\Set_{N,r}|}(\Z)$ on the vector space $\V_{N,r}$:
\[ I_{N,r}=\big( \delta\tbinom{m_1,\ldots,m_r}{n_1,\ldots,n_r} \big)_{\begin{subarray}{c} (m_1,\ldots,m_r)\in \Set_{N,r} \\ (n_1,\ldots,n_r)\in \Set_{N,r} \end{subarray}}.\]

\begin{theorem}\label{3_7}
Let $r$ be a positive integer greater than $2$ and $ F_{N,r}$ the matrix $E_{N,r}-I_{N,r}$.
Then, the following $\Q$-linear map is injective:
\begin{align*}
\W_{N,r} &\longrightarrow \ker E_{N,r}\\
p(x_1,\ldots,x_r) &\longmapsto  F_{N,r}\big(\pi(p(x_1,\ldots,x_r))\big).
\end{align*}
\end{theorem}
\begin{proof}
We first check that $ F_{N,r}\big(\pi(p)\big)\in \ker E_{N,r}$ for any polynomial $p\in \W_{N,r}$.
Define the action $\sigma_r^{(i)}$ for a polynomial $f(x_1,\dots,x_r)$ and $i\in \{ 1,2,\ldots, r-1\}$ by
\[ f(x_1,\ldots,n_r) \big| \sigma_r^{(i)}=f(x_{i+1}-x_i,x_1,\ldots,\hat{x}_{i+1},\ldots,x_r) - f(x_{i+1}-x_i,x_1,\ldots,\hat{x}_i,\ldots,x_r).\]
Note that any polynomial $p$ in $\W_{N,r}$ satisfies $p\big| (1+\sigma_r^{(1)}) =0$ (by definition), and then we have
\begin{equation}\label{eq3_9}
p(x_1,\ldots,x_r)+p(x_2,x_1,x_3,\ldots,x_r)=0.
\end{equation}
We now prove for $p\in \W_{N,r}$ 
\begin{equation}\label{eq3_10}
 p(x_1,\ldots,x_r) \big| (\sigma_r^{(j)} \sigma_r^{(i)}+\sigma_r^{(i)}\sigma_r^{(j-1)}) =0 \quad ( r-1\ge i\ge j \ge2),
\end{equation}
where $p \big| (\sigma_r^{(j)} \sigma_r^{(i)}+\sigma_r^{(i)}\sigma_r^{(j-1)})$ means $(p\big| \sigma_r^{(j)}) \big| \sigma_r^{(i)} + (p\big| \sigma_r^{(i)}) \big| \sigma_r^{(j-1)}$.
For $r-1\ge i>j \ge2$ one computes
\begin{align}
\notag&p(x_1,\ldots,x_r)  \big| \sigma_r^{(j)} \sigma_r^{(i)} \\
\label{b1}&= p(x_j-x_{j-1},x_{i+1}-x_i,x_1,\ldots,\hat{x}_j,\ldots,\hat{x}_{i+1},\ldots,x_r)\\
\label{b2}&- p(x_j-x_{j-1},x_{i+1}-x_i,x_1,\ldots,\hat{x}_{j-1},\ldots,\hat{x}_{i+1},\ldots,x_r)\\
\label{b3}&-p(x_j-x_{j-1},x_{i+1}-x_i,x_1,\ldots,\hat{x}_j,\ldots,\hat{x}_i,\ldots,x_r)\\
\label{b4}&+p(x_j-x_{j-1},x_{i+1}-x_i,x_1,\ldots,\hat{x}_{j-1},\ldots,\hat{x}_i ,\ldots,x_r),
\end{align}
and
\begin{align}
\notag&p(x_1,\ldots,x_r)  \big| \sigma_r^{(i)}\sigma_r^{(j-1)}\\
\label{b5}&=p(x_{i+1}-x_i,x_j-x_{j-1},x_1,\ldots,\hat{x}_j,\ldots,\hat{x}_{i+1},\ldots,x_r)\\
\label{b6}&-p(x_{i+1}-x_i,x_j-x_{j-1},x_1,\ldots,\hat{x}_j,\ldots,\hat{x}_i,\ldots,x_r)\\
\label{b7}&-p(x_{i+1}-x_i,x_j-x_{j-1},x_1,\ldots,\hat{x}_{j-1},\ldots,\hat{x}_{i+1},\ldots,x_r)\\
\label{b8}&+p(x_{i+1}-x_i,x_j-x_{j-1},x_1,\ldots,\hat{x}_{j-1},\ldots,\hat{x}_i,\ldots,x_r).
\end{align}
From \eqref{eq3_9} it follows $\eqref{b1}+\eqref{b5}=\eqref{b2}+\eqref{b7}=\eqref{b3}+\eqref{b6}=\eqref{b4}+\eqref{b8}=0$.
Thus we have \eqref{eq3_10} for $r-1\ge i>j \ge2$.
When $j=i$, one can check
\begin{align}
\notag&p(x_1,\ldots,x_r)  \big| \sigma_r^{(j)} \sigma_r^{(j)} \\
\label{a1}&=  p(x_j-x_{j-1},x_{j+1}-x_j,x_1,\ldots,\hat{x}_j,\hat{x}_{j+1},\ldots,x_r)\\
\label{a2}&- p(x_j-x_{j-1},x_{j+1}-x_j,x_1,\ldots,\hat{x}_{j-1},x_j,\hat{x}_{j+1},\ldots,x_r)\\
\label{a3}&-p(x_{j+1}-x_{j-1},x_{j+1}-x_j,x_1,\ldots,\hat{x}_j,\hat{x}_{j+1},\ldots,x_r)\\
\label{a4}&+p(x_{j+1}-x_{j-1},x_{j+1}-x_j,x_1,\ldots,\hat{x}_{j-1},\hat{x}_j ,\ldots,x_r),
\end{align}
and
\begin{align}
\notag&p(x_1,\ldots,x_r)  \big| \sigma_r^{(j)}\sigma_r^{(j-1)}\\
\label{a5}&=p(x_{j+1}-x_{j-1},x_j-x_{j-1},x_1,\ldots,\hat{x}_j,\hat{x}_{j+1},\ldots,x_r)\\
\label{a6}&-p(x_{j+1}-x_{j-1},x_j-x_{j-1},x_1,\ldots,\hat{x}_{j-1},\hat{x}_{j},\ldots,x_r)\\
\label{a7}&-p(x_{j+1}-x_j,x_j-x_{j-1},x_1,\ldots,\hat{x}_{j-1},x_j,\hat{x}_{j+1},\ldots,x_r)\\
\label{a8}&+p(x_{j+1}-x_j,x_j-x_{j-1},x_1,\ldots,\hat{x}_{j-1},\hat{x}_j,\ldots,x_r).
\end{align}
From \eqref{eq3_9}, it follows $\eqref{a2}+\eqref{a7}=0$.
For the relation $p\big| (1+\sigma_r^{(1)})=0$, the substitutions $x_1\rightarrow x_j-x_{j-1},x_2\rightarrow x_{j+1}-x_j, x_3\rightarrow x_1,\ldots,x_{j+1}\rightarrow x_{j-1}$ shows $\eqref{a1}+\eqref{a3}+\eqref{a5}=0$, and the substitutions $x_1\rightarrow x_{j+1}-x_{j-1},x_2\rightarrow x_{j+1}-x_j, x_3\rightarrow x_1,\ldots,x_{j}\rightarrow x_{j-2}$ together with \eqref{eq3_9} lead to $\eqref{a4}+\eqref{a6}+\eqref{a8}=0$.
Thus we have \eqref{eq3_10}.
From this, putting $\sigma_r=\sigma_r^{(1)}+\cdots +\sigma_r^{(r-1)}$ one can obtain
\begin{equation}\label{eq3_27}
\begin{aligned}
p\big| \sigma_r \big| \big(1+\sigma_r \big)&= p\big| \big( \sigma_{r}^{(2)}+\cdots + \sigma_{r}^{(r-1)}\big) \big|\sigma_r\\
&=\sum_{r-1\ge i\ge j \ge2}  p\big|\big(\sigma_{r}^{(j)}\sigma_{r}^{(i)}+\sigma_{r}^{(i)}\sigma_{r}^{(j-1)}\big)=0,
\end{aligned}
\end{equation}
where for the first equality we have used $p\big| (1+\sigma_r^{(1)}) =0$.
We now compute the coefficient of $x_1^{n_1-1}\cdots x_r^{n_r-1}$ in $p\big| \sigma_r \big| \big(1+\sigma_r \big)$.
Notice that by \eqref{eq3_1} one has $x_1^{m_1-1}\cdots x_r^{m_r-1} \big| \big(1+\sigma_r\big)=x_1^{m_1-1}\cp(x_1^{m_2-1}\cdots x_{r-1}^{m_r-1})$ (see also the proof of Lemma~\ref{3_1}).
Thus for every polynomial $p(x_1,\ldots,x_r)=\sum_{(n_1,\ldots,n_r)\in \Set_{N,r}} a_{n_1,\ldots,n_r} x_1^{n_1-1}\cdots x_r^{n_r-1}$ one can compute
\begin{align*}
&p\big| \sigma_r \big| \big(1+\sigma_r\big) \\
&= \sum_{(m_1,\ldots,m_r)\in \Set_{N,r}} a_{m_1,\ldots,m_r} \sum_{t_1+\cdots+t_r=N} \big( e\tbinom{m_1,\ldots,m_r}{t_1,\ldots,t_r} -\delta\tbinom{m_1,\ldots,m_r}{t_1,\ldots,t_r}\big) x_1^{t_1-1}\cdots x_r^{t_r-1} \big| \big(1+\sigma_r\big)\\
&=\sum_{n_1+\cdots+n_r=N} \sum_{(m_1,\ldots,m_r)\in \Set_{N,r}} a_{m_1,\ldots,m_r}\\
&\times\Big( \sum_{t_1+\cdots+t_r=N} \big( e\tbinom{m_1,\ldots,m_r}{t_1,\ldots,t_r} -\delta\tbinom{m_1,\ldots,m_r}{t_1,\ldots,t_r}\big)  e\tbinom{t_1,\ldots,t_r}{n_1,\ldots,n_r} \Big)x_1^{n_1-1}\cdots x_r^{n_r-1}.
\end{align*}
Recalling the computation in the proof of Lemma~\ref{3_1}, we have known $x_1^{m_1-1}\cdots x_r^{m_r-1} \big| \sigma_r^{(i)}=\sum_{n_1+\cdots+n_r=N} \delta\tbinom{m_2,\ldots,m_i,m_{i+2},\ldots,m_r}{n_1,\ldots,n_{i-1},n_{i+2},\ldots,n_r} b_{n_i,n_{i+1}}^{m_1} x_1^{n_1-1}\cdots x_r^{n_r-1}$, where $N=m_1+\cdots+m_r$, and hence for $1\le i,j\le r-1$
\begin{align*}
& x_1^{m_1-1}\cdots x_r^{m_r-1} \big| \sigma_r^{(j)} \big| \sigma_r^{(i)} \\
&= \sum_{n_1+\cdots+ n_r=N}\Big( \sum_{t_1+\cdots+ t_r=N} \delta\tbinom{m_2,\ldots,m_j,m_{j+2},\ldots,m_r}{t_1,\ldots,t_{j-1},t_{j+2},\ldots,t_r} b_{t_j,t_{j+1}}^{m_1}  \\
&\times\delta\tbinom{t_2,\ldots,t_i,t_{i+2},\ldots,t_r}{n_1,\ldots,n_{i-1},n_{i+2},\ldots,n_r} b_{n_i,n_{i+1}}^{t_1} \Big) x^{n_1-1}\cdots x_r^{n_r-1}.
\end{align*}
Here $\delta\tbinom{m_2,\ldots,m_j,m_{j+2},\ldots,m_r}{t_1,\ldots,t_{j-1},t_{j+2},\ldots,t_r} \delta\tbinom{t_2,\ldots,t_i,t_{i+2},\ldots,t_r}{n_1,\ldots,n_{i-1},n_{i+2},\ldots,n_r}$ is 0 if $(t_1,\ldots,t_r)\not\in \Set_{N,r}$ and $(m_1,\ldots,m_r)$, $(n_1,\ldots,n_r)\in \Set_{N,r}$.
This shows that for $(n_1,\ldots,n_r)\in \Set_{N,r}$ the coefficient of $x_1^{n_1-1}\cdots x_r^{n_r-1}$ in $p\big| \sigma_r\big| \big(1+\sigma_r\big)$ is given by 
\[ \sum_{(m_1,\ldots,m_r)\in \Set_{N,r}} a_{m_1,\ldots,m_r}\Big( \sum_{(t_1,\ldots,t_r)\in \Set_{N,r}} \big( e\tbinom{m_1,\ldots,m_r}{t_1,\ldots,t_r} -\delta\tbinom{m_1,\ldots,m_r}{t_1,\ldots,t_r}\big)  e\tbinom{t_1,\ldots,t_r}{n_1,\ldots,n_r} \Big), \]
which is exactly the $(n_1,\ldots,n_r)$-th entry of the row vector $E_{N,r}\big( F_{N,r}\big(\pi(p)\big)\big)$.
As a result, \eqref{eq3_27} implies $E_{N,r}\big( F_{N,r}\big(\pi(p)\big)\big)=0$.

Finally, we prove the injectivity of the map $ F_{N,r}\circ \pi$.
It suffices to check that 
\[ \pi\big( \W_{N,r} \big) \cap \ker  F_{N,r}=\{0\} .\]
Set $\pi(p)=(a_{n_1,\ldots,n_r})_{(n_1,\ldots,n_r)\in \Set_{N,r}}$ for $p\in \W_{N,r}$.
For each $(n_1,\ldots,n_r)\in \Set_{N,r}$, the assumption $p\in \W_{N,r}$ gives the relation
\[ \sum_{(m_1,\ldots,m_r)\in \Set_{N,r}}\Big(\delta\tbinom{m_1,\ldots,m_r}{n_1,\ldots,n_r}+ \delta\tbinom{m_3,\ldots,m_r}{n_3,\ldots,n_r} b_{n_1,n_2}^{m_1} \Big) a_{m_1,\ldots,m_r} = 0,\]
and $ F_{N,r}\big(\pi(p)\big)=0$ induces the relation
\begin{align*}
\sum_{(m_1,\ldots,m_r)\in \Set_{N,r}} & \Big( \sum_{i=1}^{r-1} \delta\tbinom{m_2,\ldots,m_i,m_{i+2}\ldots,m_r}{n_1,\ldots,n_{i-1},n_{i+2},\ldots,n_r}b_{n_i,n_{i+1}}^{m_1}\Big)  a_{m_1,\ldots,m_r}=0.
\end{align*}
Subtracting the above two relations, one has
\begin{align}\label{eq3_28}
\sum_{(m_1,\ldots,m_r)\in \Set_{N,r}} & \Big( -\delta\tbinom{m_1,\ldots,m_r}{n_1,\ldots,n_r}+ \sum_{i=2}^{r-1}  \delta\tbinom{m_2,\ldots,m_i,m_{i+2}\ldots,m_r}{n_1,\ldots,n_{i-1},n_{i+2},\ldots,n_r}b_{n_i,n_{i+1}}^{m_1}\Big)  a_{m_1,\ldots,m_r}=0.
\end{align}
Furthermore, by definition of the space $\W_{N,r}$, every relation satisfied by $a_{m_1,\ldots,m_r}$ does not depend on the choices of $(m_3,\ldots,m_r)\in \Set_{n,r-2}$ where $n=m_3+\cdots+m_r$, so that the relation \eqref{eq3_28} can be reduced to
\begin{align*}
\sum_{(m_1,m_2)\in \Set_{N-n,2}}  \Big( \delta\tbinom{m_1,\ldots,m_r}{n_1,\ldots,n_r}- \sum_{i=2}^{r-1} \delta\tbinom{m_2,\ldots,m_{i},m_{i+2},\ldots,m_r}{n_1,\ldots,n_{i-1},n_{i+2},\ldots,n_r}b_{n_i,n_{i+1}}^{m_1}\Big)  a_{m_1,\ldots,m_r}=0.
\end{align*}
Denote by $\alpha (n_1,\ldots,n_r;m_3,\ldots,m_r)$ the left-hand side of the above equation and ${\rm bij}\{ 2,\ldots,r\}$ the set of all bijections on the set $\{2,\ldots,r\} $.
Consider
\begin{align*} 
&{\mathfrak f}(m_3,\ldots,m_r):=\sum_{\sigma\in {\rm bij}\{2,\ldots,r\}} \alpha (n_1,n_{\sigma(2)},n_{\sigma(3)},\ldots,n_{\sigma(r)};m_3,\ldots,m_r)\\
&= \sum_{(m_1,m_2)\in \Set_{N-n,2}}\sum_{\sigma\in {\rm bij}\{2,\ldots,r\}} \delta\tbinom{m_1,m_2,\ldots,m_r}{n_1,n_{\sigma(2)},\ldots,n_{\sigma(r)}} a_{m_1,\ldots,m_r} \\
&- \sum_{(m_1,m_2)\in \Set_{N-n,2}}  \sum_{i=2}^{r-1}    \Big(\sum_{\sigma\in {\rm bij}\{2,\ldots,r\}} \delta\tbinom{m_2,m_3,\ldots,m_{i},m_{i+2},\ldots,m_r}{n_1,n_{\sigma(2)},\ldots,n_{\sigma(i-1)},n_{\sigma(i+2)},\ldots,n_{\sigma(r)}}b_{n_{\sigma(i)},n_{\sigma(i+1)}}^{m_1}  \Big)a_{m_1,\ldots,m_r},
\end{align*}
which is 0 for any $(m_3,\ldots,m_r)\in \Set_{n,r-2}$.
Notice that for each $\tau \in {\rm bij}\{2,\ldots,r\}$ there exists a unique $\tau' \in  {\rm bij}\{2,\ldots,r\}$ such that $\tau(j)=\tau'(j)$ for $j\in\{2,\ldots,i-1,i+2,\ldots,r\}$, and $\tau(i)=\tau'(i+1),\ \tau(i+1)=\tau'(i)$.
This pairing, together with \eqref{eqan}, provides 
\begin{align*}
& \delta\tbinom{m_2,m_3,\ldots,m_{i},m_{i+2},\ldots,m_r}{n_1,n_{\sigma(2)},\ldots,n_{\tau(i-1)},n_{\tau(i+2)},\ldots,n_{\tau(r)}}b_{n_{\tau(i)},n_{\tau(i+1)}}^{m_1} \\
&+ \delta\tbinom{m_2,m_3,\ldots,m_{i},m_{i+2},\ldots,m_r}{n_1,n_{\tau'(2)},\ldots,n_{\tau'(i-1)},n_{\tau'(i+2)},\ldots,n_{\tau'(r)}}b_{n_{\tau'(i)},n_{\tau'(i+1)}}^{m_1}  =0.
\end{align*}
Then, for each $i\in\{2,\ldots,r-1\}$, we have
\[ \Big(\sum_{\tau\in {\rm bij}\{2,\ldots,r\}} \delta\tbinom{m_2,m_3,\ldots,m_{i},m_{i+2},\ldots,m_r}{n_1,n_{\tau(2)},\ldots,n_{\tau(i-1)},n_{\tau(i+2)},\ldots,n_{\tau(r)}}b_{n_{\tau(i)},n_{\tau(i+1)}}^{m_1}  \Big)a_{m_1,\ldots,m_r}=0,\]
and hence,
\[{\mathfrak f}(m_3,\ldots,m_r) =\sum_{(m_1,m_2)\in \Set_{N-n,2}}\sum_{\tau\in {\rm bij}\{2,\ldots,r\}} \delta\tbinom{m_1,m_2,\ldots,m_r}{n_1,n_{\tau(2)},\ldots,n_{\tau(r)}} a_{m_1,\ldots,m_r}.\]
Letting $m_i= n_i$ for all $i\in \{3,\ldots,r\}$, we obtain
\begin{align*}
0&={\mathfrak f}(n_3,\ldots,n_r)\\
&=\sum_{(m_1,m_2)\in \Set_{N-n,2}} \sum_{\tau\in {\rm bij}\{2,\ldots,r\}} \delta\tbinom{m_1,m_2,n_3,\ldots,n_r}{n_1,n_{\tau(2)},n_{\tau(3)},\ldots,n_{\tau(r)}} a_{m_1,m_2,n_3,\ldots,n_r}\\
&=\Big( \sum_{\tau\in {\rm bij}\{2,\ldots,r\}} \delta\tbinom{n_3,\ldots,n_r}{n_{\tau(3)},\ldots,n_{\tau(r)}} \Big)a_{n_1,\ldots,n_r},
\end{align*}
which shows $a_{n_1,\ldots,n_r}=0$ for all $(n_1,\ldots,n_r)\in \Set_{N,r}$.
This completes the proof of Theorem~\ref{3_7}.
\end{proof}

As a corollary, the following inequality is immediate from \eqref{eq3_6}.
\begin{corollary}\label{3_8}
For each integers $r\ge3$, we have
\begin{equation*}
 \sum_{N>0} \dim_{\Q} \ker E_{N,r}x^N \ge\cusp(x)\cdot \odd(x)^{r-2} . 
\end{equation*}
\end{corollary}

{\it Remark}.
Conjecturally, the dimension of $\ker E_{N,r}$ coincides with the coefficient of $x^N$ in $\cusp(x)\cdot\odd (x)^{r-2}$.
Therefore, we may expect that $ F_{N,r}$ gives a bijection from $\pi(\W_{N,r})$ to $\ker E_{N,r}$.
We have a computational evidence up to $N=35$ for this expectation, which was checked by direct calculations using Mathematica (the verification has been completed within a few days).

%%%%%%%%%%%%%%%%%%%%%%%%%%%%%%%%%%%%%%%%%%%
\section{Proof of Theorem~\ref{1_4}}
This section is devoted to proving the inequality
\begin{equation}\label{eq4_1}
 \dim_{\Q}\ker C_{N,4} \ge \dim_{\Q} \ker E_{N,4}+\sum_{1<n<N} \bigg(\dim_{\Q} \ker  E_{N-n,3} + \rank E_{n,2} \cdot \dim_{\Q} \ker E_{N-n,2} \bigg).
\end{equation}
This inequality, together with Corollaries~\ref{3_6} and \ref{3_8}, provides
\[\sum_{N>0} \dim_{\Q} \ker C_{N,4}x^N \ge 3\cusp (x) \cdot \odd(x)^2 -\cusp(x)^2,\]
from which Theorem~\ref{1_4} follows.
Thus, it suffices to show the inequality \eqref{eq4_1}.

An essential idea for proving \eqref{eq4_1} is as follows.
It can be shown that the space of totally odd MZVs $\Zsp^{\rm odd}:=\bigoplus_{N,r\ge0}\Zsp_{N,r}^{\rm odd}$ has the structure of a bigraded $\Q$-algebra with respect to the series shuffle product developed by Hoffman~\cite{H} (e.g. $\zeta_\depth(n_1)\zeta_\depth(n_2)=\zeta_\depth(n_1,n_2)+\zeta_\depth(n_2,n_1)$).
The linear relations among totally odd MZVs of weight $N$ and depth $r$ obtained from multiplying the linear relations of $\ker \tr E_{N',r'}$ by totally odd MZVs of weight $N-N'$ and depth $r-r'$ for $(N',r')\le(N,r)$ should sit in $\ker \tr C_{N,r}$, because we are believing that $\ker \tr C_{N,r}\ (\supset \ker\tr E_{N,r})$ provide all linear relations among totally odd MZVs of weight $N$ and depth $r$.
For example, we can obtain $\dim_{\Q} \ker \tr C_{18,4}=3$, and observe that the coefficient vectors of three relations \eqref{eg3}, $\zeta_\depth(3)\times$\eqref{eg2} and $\zeta_\depth(3,3)\times$\eqref{eg1} of weight 18 and depth 4 generate the space $\ker \tr C_{18,4}$.

We prove first that multiplying by totally odd MZVs gives an injection from the space $\ker \tr E_{N',r'}$ for $(N',r')\le(N,4)$ to the space $\ker\tr C_{N,4}$ (Proposition~\ref{4_2}), and then their intersections are discussed.

%%%%%%%%%%%%%%%%%%%%%%%%%%%%%%%%%%%%%%%%%%%
\subsection{Shuffle algebra}
Since we only need the algebraic structure of $\Zsp^{\rm odd}$, consider the non-commutative polynomial algebra ${\mathbf F}$ over $\Q$ with one generator $z_{2i+1}$ in every degree $2i+1$ for $i\ge1$:
 \[ {\mathbf F}:= \Q\langle z_3,z_5,z_7,\ldots \rangle.\]
The bigraded piece of ${\mathbf F}$ of weight $N$ and depth $r$, which is the $\Q$-vector space spanned by the words $z_{n_1}\cdots z_{n_r}$ for all $(n_1,\ldots,n_r)\in \Set_{N,r}$, is denoted by $ {\mathbf F}_{N,r}$.
The empty word is regarded as $1\in\Q$.
The space ${\mathbf F}$ has the structure of a commutative, bigraded algebra over $\Q$ with respect to the shuffle product $\sh$:
 \begin{equation}\label{eq4_2}
 z_{n_1}\cdots z_{n_r} \ \sh \ z_{n_{r+1}} \cdots z_{n_{r+s}} = \sum_{\substack{\sigma\in {\mathfrak S}_{r+s}\\ \sigma(1)<\cdots <\sigma(r)\\ \sigma(r+1)<\cdots<\sigma(r+s) }} z_{n_{\sigma^{-1}(1)}} \cdots z_{n_{\sigma^{-1}(r+s)}},
 \end{equation}
where $\mathfrak{S}_n$ is the $n$-th symmetric group.
We notice that the surjective linear map from ${\mathbf F}$ to $\Zsp^{\rm odd}$ given by $z_{n_1}\cdots z_{n_r}\mapsto \zeta_{\depth}(n_1,\ldots,n_r)$ becomes an algebra homomorphism.

%%%%%%%%%%%%%%%%%%%%%%%%%%%%%%%%%%%%%%%%%%%
\subsection{Key identities} 

Define the linear map $\rho(=\rho^{(N,r)}) : \V_{N,r}\rightarrow {\mathbf F}_{N,r}$ given by
\begin{align*}
\rho : \V_{N,r} &\longrightarrow {\mathbf F}_{N,r}\\
(a_{n_1,\ldots,n_r})_{(n_1,\ldots,n_r)\in \Set_{N,r}} &\longmapsto \sum_{(n_1,\ldots,n_r)\in \Set_{N,r} } a_{n_1,\ldots,n_r} z_{n_1}\cdots z_{n_r},
\end{align*}
which by definition is an isomorphism.
For each word $w\in{\mathbf F}_{N,r}$, we define the $\Q$-linear map $\stuffle_{w}$ on the space $\V$, which corresponds to the shuffle product on ${\mathbf F}$, by the composition map
\[ 
\begin{array}{cccccccc}
\stuffle_{w}:& \V_{N',r'} &\overset{\rho}{\longrightarrow} &{\mathbf F}_{N',r'} &\overset{\sh w}{\longrightarrow}& {\mathbf F}_{N+N',r+r'} &\stackrel{\rho^{-1}}{\longrightarrow}& \V_{N+N',r+r'}\\
&v &\longmapsto& \rho(v) &\longmapsto& \rho(v)\ \sh \ w& \longmapsto &\rho^{-1} (  \rho(v)\ \sh \ w ).
\end{array}
\]
For convenience, we write $\stuffle_{n_1,\ldots,n_r}=\stuffle_{z_{n_1}\cdots z_{n_r}}$.
We note that since the algebra ${\mathbf F}$ is an integral domain, the map $\stuffle_{n_1,\ldots,n_r}$ becomes an injection.
One finds that the map $\stuffle$ provides an injection from $\ker \tr E_{N',r'}$ to $\ker \tr C_{N,4}$ for $(N',r')<(N,4)$:
\begin{proposition}\label{4_2}
(i) For each odd integer $p\ge3$, the image of $\ker\tr E_{N-p,3}$ under the injective map $\stuffle_p$ is contained in $\ker\tr C_{N,4}$:
\begin{equation*}
\stuffle_p :  \ker\tr E_{N-p,3}\longrightarrow \ker \tr C_{N,4} .
\end{equation*}
(ii) For each even integer $p\ge6$ and $(p_1,p_2)\in \Set_{p,2}$, the image of $\ker\tr E_{N-p,2}$ under the injective map $\stuffle_{p_1,p_2}$ is contained in $\ker\tr C_{N,4}$:
\begin{equation*}
  \stuffle_{p_1,p_2} : \ker\tr E_{N-p,2} \longrightarrow \ker \tr C_{N,4} .
\end{equation*}
\end{proposition}

Proposition~\ref{4_2} is shown by certain explicit formulas described in Lemma~\ref{4_1} below.
For $(p_1,\ldots,p_{r-q})\in \Set_{p,r-q}$ let us define the injective linear map $\inj_{p_1,\ldots,p_{r-q}}$ on $\V$ by
\begin{align*}
\inj_{p_1,\ldots,p_{r-q}}:\V_{N-p,q}&\longrightarrow \V_{N,r},\\
(a_{n_1,\ldots,n_{q}})_{(n_1,\ldots,n_q)\in \Set_{N-p,q}} & \longmapsto \bigl(\delta\tbinom{p_1,\ldots,p_{r-q}}{n_1,\ldots,n_{r-q}}\cdot a_{n_{r-q+1},\ldots,n_r}\bigr)_{(n_1,\ldots,n_r)\in \Set_{N,r}}.
\end{align*}
By definition \eqref{eq3_4}, for $q\in\{2,\ldots,r-1\}$ the matrix $E_{N,r}^{(q)}$ can be expressed as the direct sum of the form
\begin{align*}
 E_{N,r}^{(q)} &= \bigoplus_{\substack{1<p<N\\(p_1,\ldots,p_{r-q})\in \Set_{p,r-q}}} E_{N-p,q} \\
 &= diag(\underbrace{E_{3q,q},\ldots,E_{3q,q}}_{|\Set_{N-3q,r-q}|},\underbrace{E_{3q+2,q},\ldots,E_{3q+2,q}}_{|\Set_{N-3q-2,r-q}|},\ldots, E_{N-3(r-q),q}),
\end{align*}
and hence we have 
\begin{equation}\label{eq4_3} \ker \tr E_{N,r}^{(q)}  = \bigoplus_{\substack{1<p<N\\(p_1,\ldots,p_{r-q})\in \Set_{p,r-q}}} \inj_{p_1,\ldots,p_{r-q}} \big( \ker\tr E_{N-p,q} \big).
\end{equation}

\begin{lemma}\label{4_1}
(i) For each odd integer $p\ge3$ and $v\in \ker \tr E_{N-p,3}$, we have
\begin{equation*}
\tr E_{N,4}\bigl( \stuffle_p (v)\bigr)=\inj_p ( v) \in \ker \tr E_{N,4}^{(3)} .
\end{equation*}
(ii) For each even integer $p\ge6$, $(p_1,p_2)\in \Set_{p,2}$ and $v\in \ker \tr E_{N-p,2}$, we have
\begin{equation*}
\tr E_{N,4}^{(3)} \bigl( \tr E_{N,4}\bigl( \stuffle_{p_1,p_2} (v)\bigr)\bigr)=\sum_{(t_1,t_2)\in \Set_{p,2}}    e\tbinom{t_1,t_2}{p_1,p_2} \inj_{t_1,t_2} ( v) \in \ker \tr E_{N,4}^{(2)} .
\end{equation*}
\end{lemma}

Combining Lemma~\ref{4_1} with Proposition~\ref{3_3} (that is $\tr C_{N,4}=\tr E_{N,4}\cdot \tr E_{N,4}^{(3)}\cdot\tr E_{N,4}^{(2)}$), we have the proof of Proposition~\ref{4_2}.

Although Lemma~\ref{4_1} plays a key role in the next subsection, its proof is postponed to the Appendix because it is done by tedious computations.
In what follows, we shall try to make an explanation of the formulas in Lemma~\ref{4_1} using the notion of Brown's operator $D_m$. 
It is hoped that we obtain more general formulas from this approach (see Remark below).

Recall the graded version of Brown's operator ${\rm gr}_r^{\depth} D_m$ on the depth-graded motivic MZVs (see \S2.3).
Using the formula of $D_m$ \eqref{key1}, one can easily obtain 
\begin{equation}\label{eq4_17}
{\rm gr}^{\depth}_rD_m (\zeta^{\m}_{\depth} (n_1,\ldots,n_r) ) =\sum_{(m_1,\ldots,m_r)\in \Set_{N,r} } \delta\tbinom{m_1}{m} \ec \zeta_{m_1} \otimes \zeta^{\m}_{\depth}(m_2,\ldots,m_r).
\end{equation}

We now define the operator $d_m$ for odd $m>1$, which is a formal version of ${\rm gr}^{\depth}_rD_m$.
For $z_{n_1}\cdots z_{n_r}\in {\mathbf F}_{N,r}$, let
\[ d_m(z_{n_1}\cdots z_{n_r}) = \sum_{(m_1,\ldots,m_r)\in \Set_{N,r} } \delta\tbinom{m_1}{m} \ec z_{m_1}\otimes z_{m_2}\cdots z_{m_r}\]
and $d_m(z_n)=\delta\tbinom{m}{n}z_n\otimes 1$.
Then we obtain the $\Q$-linear map $d_m : {\mathbf F}_{N,r}\rightarrow \Q z_m\otimes {\mathbf F}_{N-m,r-1}$ by linearity.
Set $d_{<N}:=\sum_{1<m<N:{\rm odd}} d_m$:
\[ d_{<N} : {\mathbf F}_{N,r}\longrightarrow {\mathbf F}_{N,r}^{(r-1)},\]
where we write ${\mathbf F}_{N,r}^{(r-1)}=\bigoplus_{1<m<N}{\mathbf F}_{m,1}\otimes {\mathbf F}_{N-m,r-1}$.
In general, we put for $2\le q\le r$
\[{\mathbf F}_{N,r}^{(q)} = \bigoplus_{\substack{1<p<N\\(p_1,\ldots,p_{r-q})\in \Set_{p,r-q}}} \underbrace{{\mathbf F}_{p_1,1}\otimes \cdots \otimes {\mathbf F}_{p_{r-q},1}}_{r-q} \otimes {\mathbf F}_{N-p,q} ,\]
which is isomorphic to ${\mathbf F}_{N,r}$.
Notice ${\mathbf F}_{N,r}^{(r)}={\mathbf F}_{N,r}$.
Define isomorphisms $\rho^{(q)}(=\rho^{(N,r,q)})$ for $2\le q\le r$ by
\begin{align*}
\rho^{(q)} : \V_{N,r}&\longrightarrow {\mathbf F}_{N,r}^{(q)}\\
(a_{n_1,\ldots,n_r})_{(n_1,\ldots,n_r)\in \Set_{N,r}} &\longmapsto \sum_{(n_1,\ldots,n_r)\in \Set_{N,r}} a_{n_1,\ldots,n_r} z_{n_1}\otimes \cdots \otimes z_{n_{r-q}} \otimes z_{n_{r-q+1}}\cdots z_{n_r}.
\end{align*}
Note $\rho^{(r)}=\rho$.
Set $d_{<N}^{(q)}:={\rm id}^{\otimes (r-q)}\otimes d_{<N}$ for $2\le q\le r$, so then $d_{<N}^{(r)}=d_{<N}$.
By definition, for $3\le q\le r$ one immediately finds that the following diagram commutes:
 $$ \xymatrix{  
 \V_{N,r}   \ar@{->}[r]^{\rho^{(q)}} \ar@{->}[d]^{\tr E_{N,r}^{(q)}} & {\mathbf F}_{N,r}^{(q)} \ar@{->}[d]^{d_{<N}^{(q)}} \\ 
  \V_{N,r}   \ar@{->}[r]^{\rho^{(q-1)}}  & {\mathbf F}_{N,r}^{(q-1)}  }$$
We notice that the linear map $C_{N,r}$ can be written as the composition map $(\rho^{(2)})^{-1}\circ d_{<N}^{(2)} \circ \cdots \circ d_{<N}^{(r-1)}\circ d_{<N}^{(r)}\circ \rho^{(r)}$: 
\[\rho^{(2)}\big(\tr C_{N,r}(v)\big) = d_{<N}^{(2)} \circ \cdots \circ d_{<N}^{(r-1)}\circ d_{<N}^{(r)}\big(\rho^{(r)}(v)\big).\]

We are expecting that the map $d_m$ is a derivation, i.e. $d_m(w_1\ \sh \ w_2)=d_m(w_1)\ \sh \ (1\otimes w_2) + d_m(w_2)\ \sh \ (1\otimes w_1)$, however we were not able to prove this (whereas we can prove that $d_m$ is a derivation for the shuffle product of iterated integrals, since the coefficient $\ec$ is obtained from this shuffle product).
One can at least check the following (so the proof is also by boring computations, we omit it):
\begin{align*}
\label{f1}& d_m(z_p\ \sh \ z_{n_1}z_{n_2}z_{n_3} ) = d_m(z_p) \ \sh \ (1\otimes z_{n_1}z_{n_2}z_{n_3}) + d_m(z_{n_1}z_{n_2}z_{n_3}) \ \sh \ (1\otimes z_p),\\
& d_m(z_{p_1}z_{p_2}\ \sh \ z_{n_1}z_{n_2} ) = d_m(z_{p_1}z_{p_2}) \ \sh \ (1\otimes z_{n_1}z_{n_2}) + d_m(z_{n_1}z_{n_2}) \ \sh \ (1\otimes z_{p_1}z_{p_2}).
\end{align*} 
Using the first identity, we have
\begin{align*} 
\tr E_{N,4} (\stuffle_p(v)) &= \big(\rho^{(3)}\big)^{-1}\circ d_{<N} \circ \rho(\stuffle_p(v)) = \big(\rho^{(3)}\big)^{-1}\circ d_{<N} (z_p \ \sh \ \rho(v) ) \\
&=  \big(\rho^{(3)}\big)^{-1}\big( d_{<N} (z_p) \ \sh \ (1\otimes \rho(v)) + d_{<N} (\rho(v)) \ \sh \ (1\otimes z_p)\big)\\
&= \big(\rho^{(3)}\big)^{-1}\big( z_p\otimes \rho(v)\big) = \inj_p(v),
\end{align*}
where for the last equality we have used $d_{<N} (\rho(v))=0$ which is equivalent to the assumption $\tr E_{N-p,3}(v)=0$.
This gives a proof of Lemma~\ref{4_1} (i).
We can prove the formula (ii) of Lemma~\ref{4_1} in a similar way.

{\it Remark}.
If $d_m$ becomes a derivation on ${\mathbf F}$, one can prove for $v\in\ker\tr E_{N-p,r-q}$ and $(p_1,\ldots,p_q)\in \Set_{p,q}$ 
\[ \tr\big( E_{N,r}E_{N,r}^{(r-1)}\cdots E_{N,r}^{(r-q+1)} \big) \big( \stuffle_{p_1,\ldots,p_q} (v) \big) = \sum_{(t_1,\ldots,t_q)\in \Set_{p,q}} c\tbinom{t_1,\ldots,t_q}{p_1,\ldots,p_q} \inj_{t_1,\ldots,t_q} (v)  ,\]
which implies that the image of the space $\ker\tr E_{N-p,r-q}$ under the injective map $\stuffle_{p_1,\ldots,p_q} $ is contained in $\ker\tr C_{N,r}$:
\[ \stuffle_{p_1,\ldots,p_q} : \ker\tr E_{N-p,r-q} \longrightarrow \ker\tr C_{N,r}.\]

%%%%%%%%%%%%%%%%%%%%%%%%%%%%%%%%%%%%%%%%%%%%%%%%%%%%%%%%%%%%
\subsection{Proof of Theorem~\ref{1_4}}
Proposition~\ref{4_2} shows
\[ \ker\tr C_{N,4} \supset \ker\tr E_{N,4}+\sum_{1<p<N} \stuffle_p\big( \ker\tr E_{N-p,3} \big) + \sum_{\substack{1<p<N\\(p_1,p_2)\in \Set_{p,2}}} \stuffle_{p_1,p_2} \big( \ker\tr E_{N-p,2} \big). \]
For $1<p<N$ we now consider the $\Q$-vector space $\mathsf{A}_N^{(p)}\subset \V_{N,4}$ spanned by all elements of the form $\sum_{(p_1,p_2)\in \Set_{p,2}} a_{p_1,p_2}\stuffle_{p_1,p_2} (v)$, where $(a_{p_1,p_2})_{(p_1,p_2)\in \Set_{p,2}} \in {\rm Im}\, \tr E_{p,2}$ and $v\in \ker\tr E_{N-p,2}$:
\[ \mathsf{A}_N^{(p)}=\bigg\langle \sum_{(p_1,p_2)\in \Set_{p,2}} a_{p_1,p_2}\stuffle_{p_1,p_2} (v) \ \bigg| \ (a_{p_1,p_2})_{(p_1,p_2)\in \Set_{p,2}} \in {\rm Im}\, \tr E_{p,2},\ v\in \ker\tr E_{N-p,2} \bigg\rangle_{\Q}.\]
Our goal is to show the following: 
\begin{equation}\label{eq4_18}
 \ker\tr C_{N,4} \supset \ker\tr E_{N,4}\oplus \bigoplus_{1<p<N} \stuffle_p\big( \ker\tr E_{N-p,3} \big) \oplus \bigoplus_{1<p<N} \mathsf{A}_N^{(p)}.
\end{equation}
Before proving \eqref{eq4_18}, we show \eqref{eq4_1} assuming \eqref{eq4_18}.
From the injectivity of the map $\stuffle_p$, one has 
\begin{equation}\label{ts1} \dim_{\Q}\stuffle_p \big( \ker\tr E_{N-p,3} \big) = \dim_{\Q} \ker\tr E_{N-p,3}.\end{equation}
By using the fact that the algebra ${\mathbf F}$ is isomorphic to a certain polynomial algebra, one can prove the following lemma (the proof will be postponed to the Appendix).
\begin{lemma}\label{4_3}
Let ${\mathbf B}_1$ and ${\mathbf B}_2$ be subspaces of ${\mathbf F}_{N_1,2}$ and ${\mathbf F}_{N_2,2}$ respectively. 
Assume that ${\mathbf B}_1\cap {\mathbf B}_2 =\{0\}$.
Then the dimension of the space spanned by all elements composing of two products of $w_1\in{\mathbf B}_1$ and $w_2\in{\mathbf B}_2$ is equal to $\dim_{\Q} {\mathbf B}_1 \cdot \dim_{\Q} {\mathbf B}_2$:
\[ \dim_{\Q} \langle w_1\ \sh \ w_2 \mid w_1\in {\mathbf B}_1, w_2\in {\mathbf B}_2 \rangle_{\Q} =\dim_{\Q} {\mathbf B}_1 \cdot \dim_{\Q} {\mathbf B}_2.\]
\end{lemma}

From this lemma, one immediately obtains 
\begin{equation}\label{ts2} \dim_{\Q}\rho\big({\mathsf A}_N^{(p)}\big)  = \rank\tr E_{p,2}\cdot \dim_{\Q} \ker\tr E_{N-p,2}.\end{equation}
Then, the inequality \eqref{eq4_1} follows from \eqref{ts1}, \eqref{ts2} and \eqref{eq4_18}.
We now prove \eqref{eq4_18}.

\noindent
{\it Proof of \eqref{eq4_18}}.
For any element $v=\sum_{1<p<N} \stuffle_p(v_p) \in \sum_{1<p<N} \stuffle_p(\ker\tr E_{N-p,3})$, using Lemma~\ref{4_1} (i), one has 
\begin{align*}
\tr E_{N,4}(v) = \sum_{1<p<N}\tr E_{N,4}\big( \stuffle_p (v_p) \big) &= \sum_{1<p<N} \inj_p(v_p) \in \ker\tr E_{N,4}^{(3)} \\
&\bigg(= \bigoplus_{1<p<N} \inj_p\big(\ker\tr E_{N-p,3}\big), \ \mbox{by} \ \eqref{eq4_3}\bigg).
\end{align*}
Therefore, the injectivity of the map $\inj_p$ shows that $\tr E_{N,4}(v)=0$ if and only if $v_p=0$ for all $1<p<N$.
So then $v=0 \Rightarrow v_p=0$ for all $1<p<N$ implies
\[ \sum_{1<p<N} \stuffle_p(\ker\tr E_{N-p,3}) =\bigoplus_{1<p<N} \stuffle_p(\ker\tr E_{N-p,3}),\]
and $\tr E_{N,4}(v)=0 \Rightarrow v=0$ shows
\[ \ker\tr E_{N,4}\oplus \bigoplus_{1<p<N} \stuffle_p(\ker\tr E_{N-p,3}).\]

We fix a basis of $\ker\tr E_{N-p,2}$ denoted by $\{w_i^{(p)}\}_{i=1}^{k_p}$.
Then for any $v_p\in{\mathsf A}_N^{(p)}$ one can express 
\[ v_p=\sum_{i=1}^{k_p} \sum_{(p_1,p_2)\in \Set_{p,2}} a_{p_1,p_2}^{(i)} \stuffle_{p_1,p_2}\big( w_i^{(p)} \big)\]
with some elements $(a_{p_1,p_2}^{(i)})_{(p_1,p_2)\in \Set_{p,2}} \in {\rm Im}\, \tr E_{p,2}$ ($1\le i\le k_p$).
For any 
$$v=\sum_{1<p<N}\alpha_pv_p=\sum_{1<p<N}\alpha_p\sum_{i=1}^{k_p} \sum_{(p_1,p_2)\in \Set_{p,2}} a_{p_1,p_2}^{(i)} \stuffle_{p_1,p_2}\big( w_i^{(p)} \big)\in \sum_{1<p<N}{\mathsf A}_N^{(p)},$$
using Lemma~\ref{4_1} (ii), one has
\begin{align*}
\tr E_{N,4}^{(3)}\big(\tr E_{N,4} (v)\big) &= \sum_{1<p<N}\sum_{i=1}^{k_p} \sum_{(t_1,t_2)\in \Set_{p,2}} \alpha_p \bigg(\sum_{(p_1,p_2)\in \Set_{p,2}} a_{p_1,p_2}^{(i)}  e\tbinom{t_1,t_2}{p_1,p_2} \bigg) \inj_{t_1,t_2} \big( w_i^{(p)} \big)\\
&\in \ker\tr E_{N,4}^{(2)} \bigg( = \bigoplus_{\substack{1<p<N\\(t_1,t_2)\in \Set_{p,2}}} \inj_{t_1,t_2} \big( \ker\tr E_{N-p,2}\big), \ \mbox{by} \ \eqref{eq4_3} \bigg).
\end{align*}
For each $1<p<N$ the set $\{\inj_{t_1,t_2}\big(w_i^{(p)}\big)\mid 1\le i\le k_p, (t_1,t_2)\in \Set_{p,2}\}$ is linearly independent over $\Q$.
This shows that $\tr E_{N,4}^{(3)}\big( \tr E_{N,4}(v)\big) =0$ if and only if for all $1<p<N$, each coefficient of $\inj_{t_1,t_2} \big( w_i^{(p)} \big)$ for $(t_1,t_2)\in \Set_{p,2}$ and $1\le i\le k_p$, which is $\alpha_p\sum_{(p_1,p_2)\in \Set_{p,2}} a_{p_1,p_2}^{(i)}  e\tbinom{t_1,t_2}{p_1,p_2}$, is 0.
Since if $\sum_{(p_1,p_2)\in \Set_{p,2}} a_{p_1,p_2}^{(i)}  e\tbinom{t_1,t_2}{p_1,p_2}=0$ for all $(t_1,t_2)\in \Set_{p,2}$, then $0=(a_{p_1,p_2}^{(i)})_{(p_1,p_2)\in \Set_{p,2}}$ because $(a_{p_1,p_2}^{(i)})_{(p_1,p_2)\in \Set_{p,2}} \in \ker\tr E_{p,2}\cap {\rm Im}\tr E_{p,2}=\{0\}$, it turns out to be $v_p=0$ whenever $\alpha_p\neq0$.
Therefore, assuming $0=v=\sum w_p \in \sum {\mathsf A}_N^{(p)}$, we have $w_p=0$ for each $1<p<N$.
This gives
\[ \sum_{1<p<N} {\mathsf A}_N^{(p)} = \bigoplus_{1<p<N} {\mathsf A}_N^{(p)}.\]
We also have $\tr E_{N,4}^{(3)}\big(\tr E_{N,4} (v)\big)=0\Rightarrow v=0$, which is used below.
The reminder is to show that 
\[  \bigg( \ker\tr E_{N,4}\oplus \bigoplus_{1<p<N} \stuffle_{p}\big( \ker\tr E_{N-p,3}\big) \bigg) \cap\bigoplus_{1<p<N} {\mathsf A}_N^{(p)}=\{0\}.\]
It follows from the above discussion that the map $\tr E_{N,4} : \bigoplus_{1<p<N} \stuffle_{p}\big( \ker\tr E_{N-p,3}\big) \rightarrow \ker\tr E_{N,4}^{(3)}$ is an injection.
Since from the injectivity of the map $\stuffle_p$ and \eqref{eq4_3} we have
\[\dim_{\Q}\bigg(  \bigoplus_{1<p<N} \stuffle_p \big( \ker\tr E_{N-p,3} \big)\bigg)=\sum_{1<p<N}\dim_{\Q} \ker\tr E_{N-p,3} = \dim_{\Q}\ker\tr E_{N,4}^{(3)},\]
one finds that the above map becomes an isomorphism. 
Thus for $v\in \ker\tr(E_{N,4}^{(3)}\cdot E_{N,4})$ we have 
$$\tr E_{N,4}(v) \in \ker\tr E_{N,4}^{(3)} = \tr E_{N,4} \bigl( \bigoplus_{1<p<N} \stuffle_{p}\big( \ker\tr E_{N-p,3}\big) \bigr),$$
which implies that there exists $v'\in\bigoplus_{1<p<N} \stuffle_{p}\big( \ker\tr E_{N-p,3}\big)$ such that $v-v'\in \ker\tr E_{N,4}$, so that $v \in \ker\tr E_{N,4} \oplus \bigoplus_{1<p<N} \stuffle_{p}\big( \ker\tr E_{N-p,3}\big)$.
Thereby $\ker\tr(E_{N,4}^{(3)}\cdot E_{N,4}) =  \ker\tr E_{N,4}\oplus  \bigoplus_{1<p<N} \stuffle_{p}\big( \ker\tr E_{N-p,3}\big)$. 
We have already shown that if $v\in \bigoplus_{1<p<N}{\mathsf A}_N^{(p)}$ satisfies $\tr E_{N,4}^{(3)}\big( \tr E_{N,4}(v)\big)=0$, then $v=0$.
This completes the proof. \qed

\

{\it Remark}.
In general, we conjecture that the dimension of the kernel of the matrix $C_{N,r}$ is given by 
\[ \dim \ker\tr C_{N,r} \stackrel{?}{=} \dim \ker\tr E_{N,r}+\sum_{1\le q\le r-2} \Bigl(\sum_{1<p<N} {\rm rank}\ C_{p,q} \cdot \dim \ker\tr E_{N-p,r-q} \Bigr).\]
Then, one can prove the equality \eqref{eq2_5} by induction on $r$, together with the conjectural equality $\sum_{k>0} \dim \ker E_{N,r}x^k=\odd(x)^{r-2}\cusp(x)$ (or equivalently, the surjectivity of the map $ F_{N,r}$ in Theorem~\ref{3_7}).

%%%%%%%%%%%%%%%%%%%%%%%%%%%%%%%%%%%%%%%%%%%%%%%%%%%%%%
\section*{Appendix}

In this appendix, we prove Lemmas~\ref{4_1} and \ref{4_3}.

{\it Proof of Lemma~\ref{4_1}}.
Lemma~\ref{4_1} is shown by direct calculations.
We only prove (ii).
Denote $v=(a_{n_1,n_2})_{(n_1,n_2)\in \Set_{N-p,2}}$.
Then from \eqref{eq4_2} each components of the row vector $\stuffle_{p_1,p_2} (v) = (A_{n_1,\ldots,n_4}^{(p_1,p_2)})_{(n_1,\ldots,n_4)\in \Set_{N,4}}$ can be written in the form
\begin{align*}
A_{n_1,\ldots,n_4}^{(p_1,p_2)}&= \delta\tbinom{n_1,n_2}{p_1,p_2} a_{n_3,n_4} + \delta\tbinom{n_1,n_3}{p_1,p_2} a_{n_2,n_4}+\delta\tbinom{n_1,n_4}{p_1,p_2} a_{n_2,n_3} \\
&+\delta\tbinom{n_2,n_3}{p_1,p_2} a_{n_1,n_4} +\delta\tbinom{n_2,n_4}{p_1,p_2} a_{n_1,n_3} +\delta\tbinom{n_3,n_4}{p_1,p_2} a_{n_1,n_2}.
\end{align*}
The $(n_1,\ldots,n_4)$-th entry of the vector $\tr E_{N,4}\bigl(  \stuffle_{p_1,p_2}(v) \bigr)$ can be computed as follows.
\begin{align}
\notag&\sum_{(m_1,\ldots,m_4)\in \Set_{N,4}} A_{m_1,\ldots,m_4}^{(p_1,p_2)} \big( \delta\tbinom{n_1,\ldots,n_4}{m_1,\ldots,m_4}+\delta\tbinom{n_3,n_4}{m_3,m_4}b_{m_1,m_2}^{n_1} \\
\notag&+ \delta\tbinom{n_2,n_4}{m_1,m_4} b_{m_2,m_3}^{n_1} + \delta\tbinom{n_2,n_3}{m_1,m_2} b_{m_3,m_4}^{n_1}\big)\\
\notag&=\delta\tbinom{n_1,n_2}{p_1,p_2} a_{n_3,n_4} + \delta\tbinom{n_1,n_3}{p_1,p_2} a_{n_2,n_4}+\delta\tbinom{n_1,n_4}{p_1,p_2} a_{n_2,n_3} \\
\notag& + \delta\tbinom{n_2,n_3}{p_1,p_2} a_{n_1,n_4} +\delta\tbinom{n_2,n_4}{p_1,p_2} a_{n_1,n_3} +\delta\tbinom{n_3,n_4}{p_1,p_2} a_{n_1,n_2}\\
\label{c1}&+\sum_{(m_1,\ldots,m_4)\in \Set_{N,4}} \bigg( \delta\tbinom{m_1,m_2}{p_1,p_2}\delta\tbinom{n_3,n_4}{m_3,m_4} b_{m_1,m_2}^{n_1}a_{m_3,m_4}+ \delta\tbinom{m_1,m_3}{p_1,p_2}\delta\tbinom{n_3,n_4}{m_3,m_4}b_{m_1,m_2}^{n_1}a_{m_2,m_4} \\
\label{c2}&+ \delta\tbinom{m_1,m_4}{p_1,p_2}\delta\tbinom{n_3,n_4}{m_3,m_4}b_{m_1,m_2}^{n_1}a_{m_2,m_3}+\delta\tbinom{m_2,m_3}{p_1,p_2}\delta\tbinom{n_3,n_4}{m_3,m_4}b_{m_1,m_2}^{n_1}a_{m_1,m_4}\\
\label{c3}&+ \delta\tbinom{m_2,m_4}{p_1,p_2}\delta\tbinom{n_3,n_4}{m_3,m_4}b_{m_1,m_2}^{n_1}a_{m_1,m_3}+\delta\tbinom{m_3,m_4}{p_1,p_2}\delta\tbinom{n_3,n_4}{m_3,m_4}b_{m_1,m_2}^{n_1}a_{m_1,m_2}\\
\label{c4}&+ \delta\tbinom{m_1,m_2}{p_1,p_2}\delta\tbinom{n_2,n_4}{m_1,m_4}b_{m_2,m_3}^{n_1}a_{m_3,m_4}+\delta\tbinom{m_1,m_3}{p_1,p_2}\delta\tbinom{n_2,n_4}{m_1,m_4}b_{m_2,m_3}^{n_1}a_{m_2,m_4}\\
\notag&+ \delta\tbinom{m_1,m_4}{p_1,p_2}\delta\tbinom{n_2,n_4}{m_1,m_4}b_{m_2,m_3}^{n_1}a_{m_2,m_3}+\delta\tbinom{m_2,m_3}{p_1,p_2}\delta\tbinom{n_2,n_4}{m_1,m_4}b_{m_2,m_3}^{n_1}a_{m_1,m_4}\\
\label{c6}&+ \delta\tbinom{m_2,m_4}{p_1,p_2}\delta\tbinom{n_2,n_4}{m_1,m_4}b_{m_2,m_3}^{n_1}a_{m_1,m_3}+\delta\tbinom{m_3,m_4}{p_1,p_2}\delta\tbinom{n_2,n_4}{m_1,m_4}b_{m_2,m_3}^{n_1}a_{m_1,m_2}\\
\label{c7}&+ \delta\tbinom{m_1,m_2}{p_1,p_2}\delta\tbinom{n_2,n_3}{m_1,m_2}b_{m_3,m_4}^{n_1}a_{m_3,m_4}+\delta\tbinom{m_1,m_3}{p_1,p_2}\delta\tbinom{n_2,n_3}{m_1,m_2}b_{m_3,m_4}^{n_1}a_{m_2,m_4}\\
\label{c8}&+ \delta\tbinom{m_1,m_4}{p_1,p_2}\delta\tbinom{n_2,n_3}{m_1,m_2}b_{m_3,m_4}^{n_1}a_{m_2,m_3}+\delta\tbinom{m_2,m_3}{p_1,p_2}\delta\tbinom{n_2,n_3}{m_1,m_2}b_{m_3,m_4}^{n_1}a_{m_1,m_4}\\
\label{c9}&+ \delta\tbinom{m_2,m_4}{p_1,p_2}\delta\tbinom{n_2,n_3}{m_1,m_2}b_{m_3,m_4}^{n_1}a_{m_1,m_3}+\delta\tbinom{m_3,m_4}{p_1,p_2}\delta\tbinom{n_2,n_3}{m_1,m_2}b_{m_3,m_4}^{n_1}a_{m_1,m_2} \bigg).
\end{align}
One can find easily that the sum of second term of \eqref{c1} plus the sum of second term of \eqref{c2} is 0 because of \eqref{eqan}.
In the same way, we can derive $\eqref{c2}(1)+\eqref{c3}(1)=\eqref{c4}(1)+\eqref{c4}(2)=\eqref{c6}(1)+\eqref{c6}(2)=\eqref{c7}(2)+\eqref{c8}(1)=\eqref{c8}(2)+\eqref{c9}(1)=0$, where the notation, for example, \eqref{c2}(1) means the first term of \eqref{c2}.
Then the above equation can be reduced to 
\begin{align*}
&\delta\tbinom{n_1,n_2}{p_1,p_2} a_{n_3,n_4} + \delta\tbinom{n_1,n_3}{p_1,p_2} a_{n_2,n_4}+\delta\tbinom{n_1,n_4}{p_1,p_2} a_{n_2,n_3}  + \delta\tbinom{n_2,n_3}{p_1,p_2} a_{n_1,n_4} +\delta\tbinom{n_2,n_4}{p_1,p_2} a_{n_1,n_3} +\delta\tbinom{n_3,n_4}{p_1,p_2} a_{n_1,n_2}\\
&+b_{p_1,p_2}^{n_1}( a_{n_3,n_4}+a_{n_2,n_4}+a_{n_2,n_3}) + \sum_{(m_1,\ldots,m_4)\in \Set_{N,4}} \bigg( \delta\tbinom{m_3,m_4}{p_1,p_2}\delta\tbinom{n_3,n_4}{m_3,m_4} b_{m_1,m_2}^{n_1}a_{m_1,m_2} \\
&+ \delta\tbinom{m_1,m_4}{p_1,p_2}\delta\tbinom{n_2,n_4}{m_1,m_4} b_{m_2,m_3}^{n_1}a_{m_2,m_3} +  \delta\tbinom{m_1,m_2}{p_1,p_2}\delta\tbinom{n_2,n_3}{m_1,m_2} b_{m_3,m_4}^{n_1}a_{m_3,m_4} \bigg).
\end{align*}
(Note $a_{n_1,n_2}=0$ whenever $n_1+n_2\neq N-p$.)
From the assumption $\tr E_{N-p,2}(v)=0$, one has for $(n_1,\ldots,n_4)\in \Set_{N,4}$
\begin{align*}
 \sum_{(m_1,\ldots,m_4)\in \Set_{N,4}} \delta\tbinom{m_3,m_4}{p_1,p_2}\delta\tbinom{n_3,n_4}{m_3,m_4} \big(\delta\tbinom{n_1,n_2}{m_1,m_2}+ b_{m_1,m_2}^{n_1} \big)  a_{m_1,m_2}=0,
\end{align*}
which entails 
\begin{align*}
& \tr E_{N,4}\bigl(  \stuffle_{p_1,p_2}(v) \bigr) = \\
&\bigg( b_{p_1,p_2}^{n_1} (a_{n_3,n_4}+a_{n_2,n_4}+a_{n_2,n_3}) + \delta\tbinom{n_1,n_2}{p_1,p_2}a_{n_3,n_4}+\delta\tbinom{n_1,n_3}{p_1,p_2}a_{n_2,n_4}+\delta\tbinom{n_1,n_4}{p_1,p_2} a_{n_2,n_3}\bigg)_{(n_1,\ldots,n_4)\in \Set_{N,4}}.
\end{align*}
We denote by $B_{n_1,\ldots,n_4}^{(p_1,p_2)}$ the $(n_1,\ldots,n_4)$-th entry of the above.
Then the $(n_1,\ldots,n_4)$-th entry of the vector $\tr E_{N,4}^{(3)} \big( \tr E_{N,4}\bigl(  \stuffle_{p_1,p_2}(v) \bigr)  \big)$ can be computed as follows:
\begin{align}
\notag&\sum_{(m_1,\ldots,m_4)\in \Set_{N,4}} B_{m_1,\ldots,m_4}^{(p_1,p_2)}  \big( \delta\tbinom{n_1,\ldots,n_4}{m_1,\ldots,m_4}+\delta\tbinom{n_1,n_4}{m_1,m_4} b_{m_2,m_3}^{n_2} + \delta\tbinom{n_1,n_3}{m_1,m_2} b_{m_3,m_4}^{n_2} \big)\\
\label{d1}&=b_{p_1,p_2}^{n_1} (a_{n_3,n_4}+a_{n_2,n_4}+a_{n_2,n_3}) + \delta\tbinom{n_1,n_2}{p_1,p_2}a_{n_3,n_4}+\delta\tbinom{n_1,n_3}{p_1,p_2}a_{n_2,n_4}+\delta\tbinom{n_1,n_4}{p_1,p_2} a_{n_2,n_3} \\
\label{d2}&+ \sum_{(m_1,\ldots,m_4)\in \Set_{N,4}}  \delta\tbinom{n_1,n_4}{m_1,m_4} b_{m_2,m_3}^{n_2}  \big( b_{p_1,p_2}^{m_1} (a_{m_3,m_4}+a_{m_2,m_4}+a_{m_2,m_3} )\\
\label{d3}&+ \delta\tbinom{m_1,m_2}{p_1,p_2}a_{m_3,m_4} +\delta\tbinom{m_1,m_3}{p_1,p_2} a_{m_2,m_4} +\delta\tbinom{m_1,m_4}{p_1,p_2} a_{m_2,m_3} \big)\\
\label{d4}&+ \sum_{(m_1,\ldots,m_4)\in \Set_{N,4}}  \delta\tbinom{n_1,n_3}{m_1,m_2} b_{m_3,m_4}^{n_2}  \big( b_{p_1,p_2}^{m_1} (a_{m_3,m_4}+a_{m_2,m_4}+a_{m_2,m_3} )\\
\label{d5}&+ \delta\tbinom{m_1,m_2}{p_1,p_2}a_{m_3,m_4} +\delta\tbinom{m_1,m_3}{p_1,p_2} a_{m_2,m_4} +\delta\tbinom{m_1,m_4}{p_1,p_2} a_{m_2,m_3} \big).
\end{align}
We note that from $\tr E_{N-p,2}(v)=0$ one finds
\[ \eqref{d1}(2)+\eqref{d4}(1) =\sum_{(m_1,\ldots,m_4)\in \Set_{N,4}} \delta\tbinom{n_1,n_3}{m_1,m_2} b_{p_1,p_2}^{m_1}   \bigg( \delta\tbinom{n_2,n_4}{m_3,m_4} + b_{m_3,m_4}^{n_2} \bigg) a_{m_3,m_4} =0,\]
where we used the same notation as before.
We also have $\eqref{d1}(3)+\eqref{d2}(3)=0$ in a similar way.
It follows from $\tr E_{N-p,2}(v)=0$ that $\eqref{d1}(5)+\eqref{d5}(1)=\eqref{d1}(6)+\eqref{d3}(3)=0$.
Using $b_{n,n'}^m+b_{n',n}^m=0$, it follows $\eqref{d3}(1)+\eqref{d3}(2)=\eqref{d5}(2)+\eqref{d5}(3)=0$, and also changing variables $m_3\leftrightarrow m_4$ shows $\eqref{d4}(2)+\eqref{d4}(3)=0$, and changing variables $m_2\leftrightarrow m_3$ gives $\eqref{d2}(1)+\eqref{d2}(2)=0$.
So, the remainder is
\begin{align*}
& \bigg( \delta{\tbinom{n_1,n_2}{p_1,p_2}} + b_{p_1,p_2}^{n_1} \bigg)a_{n_3,n_4}= e\tbinom{n_1,n_2}{p_1,p_2}a_{n_3,n_4}=\sum_{(t_1,t_2)\in \Set_{p,2}}  e\tbinom{t_1,t_2}{p_1,p_2}\delta\tbinom{t_1,t_2}{n_1,n_2} a_{n_3,n_4},
\end{align*}
which shows 
\[\tr E_{N,4}^{(3)} \big( \tr E_{N,4} \big( \stuffle_{p_1,p_2} (v) \big)\big) = \sum_{(t_1,t_2)\in \Set_{p,2}}  e{\tbinom{t_1,t_2}{p_1,p_2}} \inj_{t_1,t_2} ( v).\]
From \eqref{eq4_3} one has $\tr E_{N,4}^{(3)} \big( \tr E_{N,4} \big( \stuffle_{p_1,p_2} (v) \big)\big) \in \ker\tr E_{N,4}^{(2)}$.
Thus we complete the proof of (ii).
$\qed$

\

For the proof of Lemma~\ref{4_3}, the important fact is that the algebra ${\mathbf F}$ is isomorphic to the polynomial algebra in the Lyndon words (see \cite[Theorem in p.589]{GC}).
Define a total ordering for the set $\{ z_3,z_5,\ldots \}$ as $z_{n_1}<z_{n_2}$ for $n_1<n_2$. 
The Lyndon word $w$ is defined as an element of $\{ z_3,z_5,\ldots \}^{\times}$ such that $w$ is smaller than every strict right factor in the lexicographic ordering $w<v$ if $w=uv$, where $u,v\neq\emptyset$.
Hereafter, the basis of ${\mathbf F}_{N,r}$ consisting of monomials in the Lyndon words is called the Lyndon basis.
For example, the Lyndon basis of the $\Q$-vector space ${\mathbf F}_{N,2}$ is given by the set $\{z_{n_1}z_{n_2} \mid (n_1,n_2)\in L_{N}\}\cup\{ z_{n_1}\ \sh \ z_{n_2}\mid (n_1,n_2)\in L^{\ast}_N \}$, where we set
\[  L_{N} := \{ (n_1,n_2)\in \Set_{N,2} \mid n_1<n_2\} \ \mbox{and} \ L^{\ast}_N:= \{ (n_1,n_2)\in \Set_{N,2} \mid n_1\le n_2\} .\]

To prove Lemma~\ref{4_3}, we have only to show the following lemma (the rest follows from the standard linear algebra).
\begin{gauss}
For integers $N_1,N_2>0$, the set $\{ \alpha_i \ \sh \ \beta_j \mid 1\le i \le g,\ 1\le j\le h\}$ is linearly independent over $\Q$, where we denote by $\{\alpha_1,\ldots,\alpha_g\}$ (resp. $\{\beta_1,\ldots,\beta_h\}$) the Lyndon basis of ${\mathbf F}_{N_1,2}$ (resp. ${\mathbf F}_{N_2,2}$).
The set $\{ \alpha_i\ \sh \ \alpha_j \mid 1\le i\le j\le g \}$ is also linearly independent.
\end{gauss}
\begin{proof}

Notice that each of the sets
\begin{align*}
\Pi_{N}&:=\{ z_{n_1}\ \sh \ z_{n_2} \ \sh \ z_{n_3} \ \sh \ z_{n_4} \mid (n_1,\ldots,n_4)\in \Set_{N,4}, n_1\le n_2\le n_3\le n_4 \},\\
\Pi_{N_1,N_2}^{(1)}&:=\{ z_{n_1}\ \sh \ z_{n_2} \ \sh \ z_{n_3}z_{n_4} \mid (n_1,n_2)\in L_{N_1}^{\ast},\ (n_3,n_4)\in L_{N_2} \},\\
\Pi_{N_1,N_2}^{(2)}& :=\begin{cases} \{ z_{n_1}z_{n_2} \ \sh \ z_{n_3}z_{n_4} \mid (n_1,n_2)\in L_{N_1},\ (n_3,n_4)\in L_{N_2} \} & \mbox{if\ }N_1\neq N_2\\
 \{ z_{n_1}z_{n_2} \ \sh \ z_{n_3}z_{n_4} \mid (n_1,n_2)\in L_{N_1},\ (n_3,n_4)\in L_{N_2},\ n_1\le n_3 \} &\mbox{if\ } N_1=N_2 \end{cases}
\end{align*}
is a subset of the Lyndon basis of ${\mathbf F}_{N,4}$.
Therefore, the set
$$\Pi_N\cup \bigcup_{(N_1,N_2)\in \Set_{N,2}} \Pi_{N_1,N_2}^{(1)}  \cup \bigcup_{(N_1,N_2)\in L_{N}^{\ast}} \Pi_{N_1,N_2}^{(2)}$$
is linearly independent over $\Q$.
The condition $n_1\le n_3$ in $\Pi_{N,N}^{(2)}$ can be explained as follows.
Let $m= {\rm max} \{ n_1 \mid (n_1,n_2)\in L_{N}\}$.
The choices of the shuffle product of two Lyndon words of depth $2$ and weight $N$ are given by the following.
 $$ 
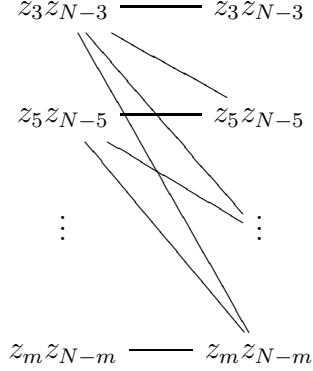
\xymatrix{  
 z_3z_{N-3}   \ar@{-}[r] \ar@{-}[rd] \ar@{-}[rdd]\ar@{-}[rddd] &  z_3z_{N-3} \\ 
z_{5}z_{N-5} \ar@{-}[r]\ar@{-}[rd] \ar@{-}[rdd] &z_{5}z_{N-5} \\
 \vdots  & \vdots \\
z_{m}z_{N-m} \ar@{-}[r] &z_{m}z_{N-m}
 }$$
Let
$$ L_{N_1,N_2} =\begin{cases} \{ (n_1,\ldots,n_4) \mid (n_1,n_2)\in L_{N_1}^{\ast}, (n_3,n_4)\in L_{N_2}^{\ast} \} & \mbox{if\ } N_1\neq N_2\\
 \{ (n_1,\ldots,n_4) \mid (n_1,n_2)\in L_{N_1}^{\ast}, (n_3,n_4)\in L_{N_2}^{\ast} ,n_1\le n_3\} & \mbox{if\ } N_1=N_2 \end{cases} $$
and
\[ \Pi_{N_1,N_2} =\{ z_{n_1}\ \sh \ z_{n_2} \ \sh \ z_{n_3} \ \sh \ z_{n_4} \mid (n_1,n_2,n_3,n_4)\in L_{N_1,N_2} \} .\] 
Then we easily find that 
\begin{align*}
\{ \alpha_i\ \sh \ \beta_j \mid 1\le i \le g , 1\le j\le h\} &= \Pi_{N_1,N_2} \cup \Pi_{N_1,N_2}^{(1)}\cup \Pi_{N_2,N_1}^{(1)} \cup \Pi_{N_1,N_2}^{(2)},\\
\{ \alpha_i\ \sh \ \alpha_j \mid 1\le i \le j\le g\} &=  \Pi_{N_1,N_1} \cup \Pi_{N_1,N_1}^{(1)}\cup \Pi_{N_1,N_1}^{(2)}.
\end{align*}
For our purpose, it only remains to verify that the set $\Pi_{N_1,N_2}$ is linearly independent, or equivalently, an overlap arising from the commutativity of the shuffle product $\sh$ doesn't occur.
This overlap comes up if we can choose $\sigma\in \mathfrak{S}_4$ such that $\sigma(l)=(n_{\sigma(1)},\ldots,n_{\sigma(4)})\in L_{N_1,N_2}$ and $\sigma(l)\neq l$ for $l=(n_1,\ldots,n_4)\in L_{N_1,N_2}$ (because this $\sigma$ gives the possible relation $z_{n_1}\ \sh \ z_{n_2} \ \sh \ z_{n_3} \ \sh \ z_{n_4}=z_{n_{\sigma(1)}}\ \sh \ z_{n_{\sigma(2)}} \ \sh \ z_{n_{\sigma(3)}} \ \sh \ z_{n_{\sigma(4)}}$ in $\Pi_{N_1,N_2}$).
However, we can check that for all $\sigma\in\mathfrak{S}_4$, if $\sigma(l)\in L_{N_1,N_2}$, then we have $\sigma(l)= l$.
We show a few cases. 
Assume $N_1\le N_2$.
For $l=(n_1,\ldots,n_4)\in L_{N_1,N_2}$, one can check that
\begin{align*}
\sigma(l)=(n_1,n_2,n_4,n_3)\in L_{N_1,N_2}&\Rightarrow n_4\le n_3 \Rightarrow n_3=n_4,\ \mbox{since}\ n_3\le n_4 \Rightarrow l=\sigma(l),\\
\sigma(l)=(n_1,n_3,n_2,n_4)\in L_{N_1,N_2}&\Rightarrow n_1+n_3=n_1=n_1+n_2 \Rightarrow n_2=n_3\Rightarrow l=\sigma(l),\\
\sigma(l)=(n_3,n_4,n_1,n_2)\in L_{N_1,N_2}&\Rightarrow n_2\le n_1 \Rightarrow n_1=n_2\Rightarrow n_3\le n_1\ \mbox{and} \ n_1\le n_3\\
&\Rightarrow n_1=n_3\Rightarrow l=\sigma(l),\\
\sigma(l)=(n_4,n_2,n_1,n_3)\in L_{N_1,N_2}&\Rightarrow n_4+n_2=n_1+n_2\Rightarrow n_1=n_4\Rightarrow n_4=n_1\le n_3\le n_4 \\
&\Rightarrow n_1=n_3=n_4 \Rightarrow l=\sigma(l).
\end{align*}
The reminder can be checked in the same way.
This completes the proof.
\end{proof}

\section*{Acknowledgement}
This paper is a part of the author's doctoral dissertation submitted to Kyushu University (2014).
The author is grateful to Professor Masanobu Kaneko and Professor Seidai Yasuda for initial advice and many useful comments over the course of this work.
His thanks are also directed to Professor Kentaro Ihara who pointed out factual mistakes in the early paper.
Without his patient help, this paper could not be completed. 
The author thanks the referees and Professor Hidekazu Furusho for careful reading of the manuscript and useful comments.
A special thank is also extended to Francis Brown.

%%%%%%%%%%%%%%%%%%%%%%%%%%%%%%%%%%%%%%%%%%

\end{document}